\documentclass[12pt,reqno,a4paper]{article}
\usepackage{fullpage,amsthm,amssymb,amsmath}
\usepackage{comment}
\usepackage{euscript}
\usepackage{xcolor}

\newtheorem{theorem}{Theorem}
\newtheorem{proposition}[theorem]{Proposition}
\newtheorem{lemma}{Lemma}
\newtheorem{cor}{Corollary}
\newtheorem{remark}{Remark}

\renewcommand{\epsilon}{\varepsilon}

\def\N{\mathbb{N}}
\def\Z{\mathbb{Z}}
\def\R{\mathbb{R}}
\def\om{\omega}

\DeclareMathOperator{\esssup}{esssup}
\DeclareMathOperator{\essinf}{essinf}
\DeclareMathOperator{\Cov}{Cov}
\DeclareMathOperator{\BV}{BV}

\author{D. Dragi\v cevi\' c \footnote{Department of Mathematics, University of Rijeka, Rijeka Croatia. {\tt E-mail: ddragicevic@math.uniri.hr}.}\and 
	Y. Hafouta\footnote{Department of Mathematics, The University of Florida, USA. {\tt E-mail: yeor.hafuta@ufl.edu}.}}

\begin{document}

\title{Iterated invariance principle for random dynamical systems}
\maketitle

\begin{abstract}
 We prove a weak iterated invariance principle for a large class of non-uniformly expanding random dynamical systems. In addition, we give a quenched homogenization result for fast-slow systems in the case when the fast component corresponds to a uniformly expanding random system. 
 Our techniques rely on the appropriate martingale decomposition. 
\end{abstract}
\section{Introduction}

A very important discovery made in the previous century is that many chaotic  deterministic dynamical systems satisfy the central limit theorem (CLT), where the chaoticity usually corresponds to some form of hyperbolicity (uniform, nonuniform or partial).  Since then many other classical results in probability theory were extended to deterministic dynamical systems, 
including  the weak invariance principle (WIP) which represents 
the functional version of the CLT. 

More recently, there has been a growing interest in the so-called iterated weak invariance principle, which concerns the asymptotic behaviour  of random functions of the form
$
\mathbb W_{k,n}(t)=n^{-k/2}\mathbb W_{k,[nt]},
$
where
$$
\mathbb W_{k,n}=\sum_{0\leq i_1\leq i_2<\ldots <i_k<n}X_{i_1}\otimes X_{i_1}\otimes\cdots \otimes X_{i_k}
$$
for several classes of zero-mean vector-valued stationary processes $(X_j)_{j\geq 0}$. Expressions of the form $\mathbb W_{k,n}$
 are a  special type of local statistics and they arose recently 
in works related to rough
path theory, data science and machine learning (see \cite{[11], [8], [20]}).
We refer to \cite{Friz-Kifer, KM, Kifer 1, Kifer 2} for results in this direction. Note that in \cite{KM} weak convergence was obtained, while in the other papers strong approximations were obtained. Namely they concern coupling of the iterated sums  $\mathbb W_{k,n}(t)$ with their corresponding limiting Gaussian  processes with almost sure  estimates on the error terms (or estimates in $L^p$ for an appropriate $p$).

In the context of deterministic dynamical systems, the primary interest for studying the iterated weak invariance principle comes from its role in the \emph{homogenization} results for multiscale fast-slow systems, which provide sufficient conditions under which solutions of such systems converge (in an appropriate sense) towards a solution of certain stochastic differential equations. It turns out that if the iterated weak invariance principle  (associated to the fast or chaotic component of the system) holds and under appropriate moment bounds, the machinery of rough path theory yields appropriate homogenization results. This program has been initiated by I. Melbourne and collaborators, and has so far produced a number of important results (see~\cite{CFKM, CFKMZ0, CFKMZ, EGK, GM, GM1, KM, KM2, KKM} and references therein). We would also like to to refer to \cite{Friz-Kifer} for corresponding  almost sure diffusion  approximations.


From a physical point of view stationary processes are less natural, since very often external forces (or noise) are involved.
In the setup of this paper this leads to non-stationary dynamical systems which are formed by compositions of different maps.
A random dynamical system is a special case of a non-autonomous system, where the noise is modeled by a probability preserving preserving system (i.e. a stationary processes). That is, let $(\Omega,\mathcal F,\mathbb P)$ be a probability space and let $\sigma:\Omega\to \Omega$ be an invertible ergodic probability preserving transformation. Then,  the random dynamics is formed by compositions of a family of maps $T_\om$, $\om\in\Omega$  along  $\sigma$-orbit of a point $\omega$ so that the $n$-th step iterate of the system is given by 
$$
T_\omega^{(n)}=T_{\sigma^{n-1}\omega}\circ\ldots\circ T_{\sigma\om}\circ T_{\omega}.
$$
Ergodic theory of random dynamical systems has been extensively studied in the past decades, with applications to economics,  statistical
physics and meteorology (see \cite{Kifer 1986}). In recent years, a major attention was devoted to limit theorems for random expanding (or hyperbolic) systems. In this context, the process 
 $(X_j)_j$ has the form $X_j=f_{\sigma^j\omega}\circ T_\omega^{(j)}$ for a fixed $\om\in\Omega$ which belongs to a set of full $\mathbb P$-probability and for wide classes of non-uniformly expanding random dynamical systems $(T_\omega)_{\omega \in \Omega}$.  Moreover, $(f_\omega)_\omega$ is a suitable class of observables. 
For uniformly expanding systems such results  include the CLT, CLT with rates, local limit theorems and almost sure invariance principles
for $\mathbb W_{1,n}(1)$ for many classes of random expanding or hyperbolic maps. We refer the readers to \cite{BB, DFGTV1, DFGTV2, DFGTV3, DH, DH1, DH2, DH3, HK, YH YT} for a partial list of such results.

The results described above hold true for uniformly expanding/hyperbolic random dynamical systems. The CLT and related results  for non-uniformly random i.i.d hyperbolic systems were studied in \cite{ANV, ALS}. Using the independence of the maps $(T_{\sigma^j\om})_{j\geq 0}$ these results rely on the spectral gap of the associated (deterministic) annealed operator. This approach fails when the maps are not independent since then the iterates of the annealed operator are no longer related to the stochastic behaviour of the random dynamical system.

For non-independent maps and non-uniformly expanding random systems the situation is more complicated. 
Recently two approaches were developed. In \cite{DHS, DS},  a scaling approach was introduced for general ergodic random environments. The scaling condition reads  as $\text{esssup}_{\om\in\Omega}(\|f_\om\|K(\om))<\infty$ for an appropriate tempered random variable $K(\om)$. However, in that generality the sufficient conditions (related to observables) for the CLT are harder to verify since $K(\om)$ comes from Oseledets multiplicative ergodic  theorem and it is not computable. In~\cite{DHS} it is showed that in such generality some scaling condition is needed.
Recently, in~\cite{YH} and~\cite{YH 2023} a different approach  was developed for random environments $(\Omega,\mathcal F,\mathbb P,\sigma)$ with some amount of weak dependence/mixing. In particular, in \cite{YH} the second author provided explicit examples where Kifer's inducing approach \cite{Kifer 1998} yields verifiable conditions for a variety of limit theorems.

The main objective of the present paper is to establish the  iterated weak invariance principle for $\mathbb W_{2,n}(t)$, when $X_j$ has the form $X_j(x)=f_{\sigma^j\omega}\circ T_\omega^{(j)}$ for a fixed $\om\in\Omega$ which belongs to a set of full $\mathbb P$-probability and for wide classes of non-uniformly expanding random dynamical systems. We stress that all the  results are new already in the  uniformly expanding case.
Our approach relies on the martingale method for establishing limit theorems and follows closely the arguments developed in~\cite{KM}. However, we stress that the nonuniformity of dynamics (with respect to the random parameter) requires nontrivial modifications of the approach in~\cite{KM}, starting essentially with the construction of appropriate martingale decomposition.  As an application of the iterated WIP we obtain appropriate homogenization result (discussed above) for uniformly expanding random maps. In the  non-uniform case at the present moment it is unclear how to apply rough path theory since the estimates of $\|\mathbb W_{k,n}(t)-\mathbb W_{k,n}(s)\|_{L^p}$ (for $k=1,2$) are not uniform in $\om$, and so standard tightness criteria in H\"older norms needed to apply rough path theory might fail, and our result only yields the convergence of the finite dimensional distributions. We refer to Remark~\ref{failure} for details.

\section{Preliminaries}\label{P}

We begin by introducing our setup. Let $(\Omega, \mathcal F, \mathbb P)$ be a  probability space and  $\sigma \colon \Omega \to \Omega$ an invertible transformation preserving $\mathbb P$ such that the system $(\Omega, \mathcal F, \mathbb P,\sigma)$ is ergodic. 

Moreover, let $M$ be a metric space equipped with the Borel $\sigma$-algebra $\mathcal B$.
In addition, let $T_\omega\colon M\to M$, $\omega \in \Omega$ be a family of non-singular maps on $M$. Note that in principle we can consider also maps $T_\omega:\mathcal E_\omega\to \mathcal E_{\sigma\omega}$ acting on random subspaces of $M$, but for the sake of simplicity we will focus on the case when all $\mathcal E_\omega$ coincide with $M$.
 For $\omega \in \Omega$ and $n\in \N$, set
\[
T_\omega^{(n)}:=T_{\sigma^{n-1}\omega}\circ \ldots \circ T_\omega. 
\]
We assume that there exists a family $(\mu_\om)_{\om \in \Omega}$ of Borel probability measures on $M$ which is equivariant, i.e.
\begin{equation}\label{equi}
(T_\om)_*\mu_\om=\mu_{\sigma\om}, \quad \text{for $\mathbb P$-a.e $\om \in \Omega$.}
\end{equation}
Let $L_\omega$ be the transfer operator of $T_\omega$, namely for a bounded function $g$, $L_\om g$ is the density of the measure $(T_\om)_*(gd\mu_\om)$ with respect to $\mu_{\sigma\omega}$. Then $L_\omega$ satisfies the following duality relation
\begin{equation}\label{Duality}
 \int_{M} f\cdot (g\circ T_\omega)\, d\mu_\omega=\int_{M}(L_\omega f)g\, d\mu_{\sigma\omega},   
\end{equation}
for all bounded and measurable functions $f,g$ on $M$.
We assume that for $\mathbb P$-a.e. $\omega \in \Omega$, $L_\omega$ acts as a bounded linear operator on a certain Banach space $(\mathcal H, \| \cdot \|_{\mathcal H})$ consisting of real-valued observables on $M$ with the properties that $\mathcal H$ contains constant functions on $M$ and that
\[
\|\varphi\|_{L^\infty(\mu_\om)}\le \|\varphi\|_{\mathcal H}, \quad \varphi \in \mathcal H.
\]
The above requirement ensures that $\|\cdot \|_{\mathcal H}$ dominates $\| \cdot \|_{L^p(\mu_\omega)}$ for every $\omega \in \Omega$ and $1\le p\le \infty$.
Assume further that there exist random variables $K\colon \Omega \to [1, \infty)$ and $A_n\colon \Omega \to (0,\infty)$  such that
\begin{equation}\label{NewExpConv}
\left \|L_\om^{(n)}\varphi-\int_M \varphi \, d\mu_\om \right \|_{L^\infty( \mu_{\sigma^n\om})}\leq K(\om)A_n(\om)\|\varphi \|_{\mathcal H} \quad \text{for $\mathbb P$-a.e. $\omega \in \Omega$, $\varphi \in \mathcal H$ and $n\in \N$,}
\end{equation}
where 
\[
L_\omega^{(n)}:=L_{\sigma^{n-1} \omega}\circ \ldots \circ L_\omega.
\]
In addition, we require that for some $q_0\ge 4$  we have that
\begin{equation}\label{rho norm}
\sum_{j=1}^\infty \|A_j\|_{L^{q_0}(\Omega, \mathcal F, \mathbb P)}<\infty.
\end{equation}

Let us note the following simple consequence of~\eqref{NewExpConv} which gives a quenched decay of correlation result.
\begin{lemma}\label{dec}
For $\mathbb P$-a.e. $\omega \in \Omega$, $n\in \N$, $\varphi \in \mathcal H$ and $\psi \in L^1(\mu_{\sigma^n \omega})$, we have that 
\[
\bigg{\lvert} \int_M \varphi\cdot( \psi \circ T_{\om}^{(n)}) \, d\mu_{\omega} -\int_M\varphi \, d\mu_\omega \cdot \int_M \psi \,  d\mu_{\sigma^n \omega} \bigg{\rvert} \le K(\omega)A_n(\omega) \lVert \psi \rVert_{L^1(\mu_{\sigma^n \omega})}\cdot \lVert \varphi \rVert_{\mathcal H}.
\]
\end{lemma}
\begin{proof}
We have that 
\[
\begin{split}
&\bigg{\lvert} \int_M \varphi\cdot( \psi \circ T_{\om}^{(n)}) \, d\mu_{\omega} -\int_M \varphi \, d\mu_\omega \cdot \int_M \psi \,  d\mu_{\sigma^n \omega} \bigg{\rvert} \\
&=
\bigg{\lvert}\int_M ( L_\omega^{(n)}\varphi)\psi \, d\mu_{\sigma^n\omega}-\int_M\varphi \, d\mu_\omega \cdot \int_M \psi \,  d\mu_{\sigma^n \omega} \bigg{\rvert} \\
&=\bigg{\lvert}\int_M(L_\omega^{(n)}\varphi-\int_M \varphi\, d\mu_\omega)\psi \,  d\mu_{\sigma^n \omega} \bigg{\rvert} \\
&\le \left \|L_\omega^{(n)}\varphi-\int_M\varphi\, d\mu_\omega \right \|_{L^{\infty}(\mu_{\sigma^n\om})} \cdot \lVert \psi \rVert_{L^1(\mu_{\sigma^n \omega})}.
\end{split}
\]
The desired conclusion now follows readily from~\eqref{NewExpConv}.
\end{proof}

We will also consider the skew-product transformation $\tau \colon \Omega \times M\to \Omega \times M$ defined by
\begin{equation}\label{tau}
\tau(\omega, x)=(\sigma \omega, T_\omega (x)), \quad (\omega, x)\in \Omega \times M.
\end{equation}
Let $\mu$ be a measure on $\Omega \times M$ such that
\begin{equation}\label{egmu}
\mu(A\times B)=\int_A \mu_\omega (B)\, d\mathbb P(\omega), \quad \text{for $A\in \mathcal F$ and $B\in \mathcal B$.}
\end{equation}
Then, $\mu$ is a $\tau$-invariant probability measure. In the sequel, we will assume that $\mu$ is ergodic; for sufficient conditions that ensure this we refer to Proposition~\ref{ergodicity} given in the Appendix.
We can now establish the following annealed decay of correlation result.
\begin{lemma}\label{Skew DEC}
Let $p_1, p_2, r>0$ be  such that $\frac{1}{q_0}+\frac{1}{r}+\frac{1}{p_1}+\frac{1}{p_2}\le 1$ and 
suppose that $K\in L^{r}(\Omega,\mathcal F,\mathbb P)$. Let $\Phi, \Psi \colon \Omega \times M\to \R$ be measurable maps satisfying the following conditions:
\begin{itemize}
\item either $\int_M \Phi(\omega, \cdot) \, d\mu_\omega=0$ for $\mathbb P$-a.e. $\omega \in \Omega$ or $\int_M \Psi(\omega, \cdot) \, d\mu_\omega=0$ for $\mathbb P$-a.e. $\omega \in \Omega$;
\item $\Phi (\omega, \cdot)\in \mathcal H$ and $\Psi(\omega, \cdot)\in L^1(\mu_\omega)$ for $\mathbb P$-a.e. $\omega \in \Omega$;
\item $F\in L^{p_1}(\Omega, \mathcal F, \mathbb P)$ and $G\in L^{p_2}(\Omega, \mathcal F, \mathbb P)$, where
$F(\omega):=\|\Phi(\omega, \cdot)\|_{\mathcal H}$ and $G(\omega):=\|\Psi(\omega, \cdot)\|_{L^1(\mu_\omega)}$.
\end{itemize}
Then, for $n\in \N$ we have that 
$$
\left|\int_{\Omega \times M} \Phi \cdot (\Psi \circ \tau^n)d\mu\right|\leq \|K\|_{L^{r}(\Omega, \mathcal F, \mathbb P)}\cdot \|F\|_{L^{p_1}(\Omega, \mathcal F, \mathbb P)}\cdot \|G\|_{L^{p_2}(\Omega, \mathcal F, \mathbb P)} \cdot
\|A_n\|_{L^{q_0}(\Omega,\mathcal F, \mathbb P)}.
$$ 
\end{lemma}

\begin{proof}
By Lemma~\ref{dec}, we have that 
\[
\begin{split}
\left |\int_{\Omega \times M} \Phi \cdot (\Psi \circ \tau^n)\,d\mu \right | &= \left | \int_\Omega \left(\int_M \Phi(\om, \cdot) \cdot (\Psi(\sigma^n\om, \cdot)\circ T_\om^{(n)})\, d\mu_\om\right)\, d\mathbb P(\om)\right |\\
&\le \int_\Omega K(\omega)A_n(\omega)F(\omega)G(\sigma^n \omega)\, d\mathbb P(\omega).
\end{split}
\]
Hence, since $\sigma$ preserves $\mathbb P$, the desired conclusion follows by applying the H\"older inequality.
\end{proof}

\begin{remark}\label{REM}
Let $\mathcal H^e$ denote the space consisting of all $\varphi=(\varphi_1, \ldots, \varphi_e)\colon X\to \R^e$ such that $\varphi_i \in \mathcal H$ for $1\le i\le e$. Then, $\mathcal H^e$ is a Banach space with respect to the norm 
$\lVert \varphi \rVert=\max_{1\le i\le e}\lVert \varphi_i\rVert_{\mathcal H}$.  We can now extend each $L_\omega$ to the bounded operator on $\mathcal H^e$. More precisely, for $\varphi =(\varphi_1,\ldots, \varphi_e)\in \mathcal H^e$, we set
\[
L_\omega^e \varphi=(L_\omega \varphi_1, \ldots, L_\omega \varphi_e),
\]
for $\omega \in \Omega$. Then, \eqref{NewExpConv}  immediately extends to the compositions of $L_\omega^e$.
In order  to keep the notation as simple as possible, in the rest of the paper instead of $\mathcal H^e$ and $L_\omega^e$ we will write $\mathcal H$ and   $L_\omega$, respectively.
\end{remark}

\section{The asymptotic variance, the CLT and the law of iterated logarithm (LIL)}
We  first formulate sufficient conditions under which we have (quenched) central limit theorem (CLT)  and law of iterated logarithm  (LIL) for a suitable class of observables. The following result is of independent interest but it will also play an important role in our results devoted to the iterated weak invariance principle.

\begin{theorem}\label{CLT}
Let  \eqref{rho norm} hold with some $q_0\geq 4$   and suppose that $K\in L^r(\Omega,\mathcal F, \mathbb P)$ for some $r\geq \frac{2q_0}{q_0-2}$. 
Let $u\colon \Omega \times M\to \R^e$ be a measurable map such that $u_\omega \in \mathcal H$ and $\int_M u_\omega \, d\mu_\omega=0$ for $\mathbb P$-a.e. $\omega \in \Omega$, where $u_\omega:=u(\omega, \cdot)$. In addition, suppose that the random variable 
$\omega\to \|u_\omega\|_{\mathcal H}$ belongs to $L^p(\Omega,\mathcal F, \mathbb P)$ for some $p$ such that \[\frac{1}{p}+\frac{1}r+\frac{1}{q_0}\leq \frac12 \quad \text{and} \quad  \frac{2}p+\frac{1}r+\frac{1}{q_0}\leq 1.\] 
Consider the functions
$$
S_n^\omega u=\sum_{j=0}^{n-1}u_{\sigma^j\omega}\circ T_\omega^{(j)},
$$ 
as random variables on the probability space $(M,\mu_\omega)$.  Then, the following holds:

(i) there exists a positive semi-definite matrix $\Sigma^2$ so that for $\mathbb P$-a.e. $\omega \in \Omega$ we have that
$$
\Sigma^2=\lim_{n\to \infty}\frac{1}n\text{Cov}_{\mu_\omega}(S_n^\omega u).
$$
 Moreover, $\Sigma^2$ is not  positive definite if and only if there is a unit vector $v\in \R^e$ such that 
 $v\cdot u=q-q\circ \tau$  for some measurable function $q\colon \Omega \times M\to \R$ satisfying $q\in L^2(\Omega \times M, \mu)$, where $(v\cdot u)(\omega, x)=v\cdot u(\omega, x)$, $(\omega, x)\in \Omega\times M$ and $\cdot$ denotes the scalar product on $\R^e$;
 \vskip0.1cm
  (ii) for $\mathbb P$-a.e. $\omega \in \Omega$, the sequence $S_n^\om u$ obeys the CLT, i.e.
  the sequence of random variables 
  $n^{-1/2}S_n^\om u$  converges  in distribution (on the probability space $(M,\mathcal B,\mu_\omega)$) to a zero mean multivariate Gaussian distribution whose covariance matrix is $\Sigma$;  
\vskip0.1cm
(iii) for $\mathbb P$-a.e. $\omega \in \Omega$,
we have that 
$$
S_n^\om u=O(\sqrt{n \ln\ln n}),\quad \text{$\mu_\omega$-a.s.}
$$
\end{theorem}

\begin{proof}
We first claim that  it is enough to prove the theorem in the one-dimensional case when $e=1$. Indeed, suppose that the result holds for $e=1$. Then, the third assertion follows by applying it to $u^j$ for each $1\le j\le e$, where $u=(u^1, \ldots, u^e)$.  The proof of the first assertion also follows from the one-dimensional case. Indeed, for a real-valued function $\tilde u$ satisfying the conditions of the theorem,  set
$$
\Sigma^2(\tilde u)=\lim_{n\to\infty}\frac 1n\text{Var}(S_n^\omega \tilde u).
$$
Define $\Sigma_{i,j}^2=\frac12(\Sigma^2(u^{i}+u^{j})-\Sigma^2(u^{i})-\Sigma^2(u^{j}))$  for $1\le i, j \le e$ and 
let $\Sigma^2=(\Sigma_{i,j}^2)_{1\leq i,j\leq e}$  Then, for $\mathbb P$-a.e. $\omega \in \Omega$ we have that
$$
\Sigma^2=\lim_{n\to \infty}\frac{1}n\text{Cov}_{\mu_\omega}(S_n^\omega u),
$$
and in addition 
$$
v^t\Sigma^2 v=\lim_{n\to\infty}\text{Var}(S_n^\omega (v\cdot u)),
$$
for every $v\in \R^e$.
Thus $\Sigma^2$ is not  positive  definite if and only if  the function $v\cdot u$ is a coboundary for some unit vector $v$. This shows that the first assertion of the theorem follows from the scalar case. To derive the second assertion (CLT), it is enough to show that all linear combinations of the finite dimensional distributions  converge to a zero-mean normal random variable with variance  $v^t\Sigma v$. However, this follows from the CLT in the scalar case applied to the function $v\cdot u$.

Let us now prove the theorem in the case when $e=1$.
Our goal is to apply~\cite[Theorem 2.3]{Kifer 1998} with the trivial set $Q=\Omega$, namely when there is no actual inducing involved. This requires us to verify the following three conditions:

\begin{equation}\label{Ver1 2}
c\in L^2(\Omega, \mathcal F, \mathbb P), \quad \text{where $c(\omega):=\|u_\omega\|_{\mathcal H}$ for $\omega \in \Omega$,}
\end{equation}

\begin{equation}\label{Ver2 2}
\left\|\sum_{n=0}^\infty|\mathbb E_{\mu_\omega}[u_\om\cdot u_{\sigma^n\om}\circ T_\omega^{(n)}]|\right\|_{L^1(\Omega, \mathcal F,\mathbb P)}<\infty
\end{equation}
and
\begin{equation}\label{Ver3 2}
\left\|\sum_{n=0}^{\infty}\mathbb E_{\mu_\omega}(|L_{\sigma^{-n}\omega}^{(n)} u_{\sigma^{-n}\omega}|)\right\|_{L^2(\Omega, \mathcal F, \mathbb P)}<\infty.
\end{equation}   
Condition \eqref{Ver1 2} follows from the assumptions of the theorem. 
To show that condition \eqref{Ver2 2} is in force, let us fix some $n\in \N$. We simplify the notation by writing $\|\cdot \|$ instead of $\|\cdot\|_{\mathcal H}$. We first note that Lemma~\ref{dec} gives that 
$$
|\mathbb E_{\mu_\omega}[u_\omega\cdot u_{\sigma^n\omega}\circ T_\omega^{(n)}]|\leq  K(\omega)A_n(\omega)\|u_\omega\| \cdot \|u_{\sigma^n\omega}\|_{L^1(\mu_{\sigma^n\omega})}, \quad \text{{for $\mathbb P$-a.e. $\omega \in \Omega$.}}
$$
Therefore, by the H\"older inequality and the $\sigma$-invariance of $\mathbb P$,
$$
\left\||\mathbb E_{\mu_\omega}[u_\omega\cdot u_{\sigma^n\omega}\circ T_\omega^{(n)}]|\right\|_{L^1(\Omega, \mathcal F, \mathbb P)}\leq \|c\|_{L^p(\Omega, \mathcal F,\mathbb P)}^2\|K\|_{L^r(\Omega, \mathcal F, \mathbb P)}\left\|A_n\right\|_{L^{q_0}(\Omega, \mathcal F,\mathbb P)}.
$$
Thus, \eqref{rho norm} gives that 
$$
\left\|\sum_{n=1}^\infty|\mathbb E_{\mu_\omega}[u_\om\cdot  u_{\sigma^n\om}\circ T_\omega^n]|\right\|_{L^1(\Omega, \mathcal F, \mathbb P)}\leq 
\|c\|_{L^p(\Omega, \mathcal F, \mathbb P)}^2\|K\|_{L^r(\Omega, \mathcal F, \mathbb P)}\sum_{n= 1}^\infty \left\|A_n\right\|_{L^{q_0}(\Omega, \mathcal F, \mathbb P)}<\infty.
$$
Hence, \eqref{Ver2 2} holds.
Next, we verify \eqref{Ver3 2}. We have  (see~\eqref{NewExpConv}) that
$$
\left|\mathbb E_{\mu_\omega}(|L_{\sigma^{-n}\omega}^{(n)} u_{\sigma^{-n}\omega}|)\right|\leq A_n(\sigma^{-n}\omega)c(\sigma^{-n}\omega)K(\sigma^{-n}\om), \quad \text{for $\mathbb P$-a.e. $\omega \in \Omega$.}
$$
Thus, by the H\"older inequality  and the $\sigma$-invariance of $\mathbb P$,  we see that
$$
\left\|\sum_{n=1}^{\infty}\mathbb E_{\mu_\omega}(|L_{\sigma^{-n}\omega}^{(n)} u_{\sigma^{-n}\omega}|)\right\|_{L^2(\Omega, \mathcal F, \mathbb P)}\leq 
\|c\|_{L^p(\Omega, \mathcal F, \mathbb P)}\|K\|_{L^r(\Omega,\mathcal F,\mathbb P)}\sum_{n= 1}^\infty \|A_n\|_{L^{q_0}(\Omega, \mathcal F, \mathbb P)}<\infty,
$$
which yields~\eqref{Ver3 2}.
\end{proof}

\section{Martingale decomposition}
Let $v\colon \Omega \times M\to \R^e$ be a measurable map satisfying the following properties:
\begin{itemize}
\item for $\mathbb P$-a.e. $\omega \in \Omega$, \begin{equation}\label{spaceH} v_\omega:=v(\omega, \cdot)\in \mathcal H;\end{equation}
\item for $\mathbb P$-a.e. $\omega \in \Omega$, \begin{equation}\label{centering}
\int_M v_\omega \, d\mu_\omega=0.
\end{equation}
\end{itemize}
The first requirement says (see Remark~\ref{REM}) that the coordinate functions of $v_\omega$ belong to $\mathcal H$, while the second requirement implies that our observable $v$ is fiberwise centered with respect to the family $(\mu_\omega)_{\omega \in \Omega}$ of equivariant measures (see~\eqref{equi}).
For $\omega \in \Omega$, set
\begin{equation}\label{chi}
\chi_\omega:=\sum_{n=1}^\infty L_{\sigma^{-n} \omega}^{(n)}(v_{\sigma^{-n} \omega}),
\end{equation}
and
\begin{equation}\label{mdec}
m_\omega=v_\omega+\chi_\omega-\chi_{\sigma \omega}\circ T_\omega. 
\end{equation}
Provided that $\chi_\omega \in L^\infty (\mu_\omega)$ for $\mathbb P$-a.e. $\omega \in \Omega$, it follows  that $m_\omega \in L^\infty(\mu_\omega)$ for $\mathbb P$-a.e. $\omega \in \Omega$. In that case we can also consider $m, \chi \colon \Omega \times M\to \R^e$ given by $m(\omega, x)=m_\omega(x)$ and $\chi(\omega, x)=\chi_\omega (x)$ for $(\omega, x)\in \Omega \times M$.
Before formulating conditions that will ensure that $\chi_\omega$ is well-defined, we point out  few important observations.
\begin{lemma}
Suppose that $\chi_\omega \in L^\infty(\mu_\omega)$ for $\mathbb P$-a.e. $\omega \in \Omega$. Then, 
\begin{equation}\label{lm}
L_\omega (m_\omega)=0, \quad \text{for $\mathbb P$-a.e. $\omega \in \Omega$.}
\end{equation}
\end{lemma}

\begin{proof}
By~\eqref{chi} and~\cite[Lemma 7]{DFGTV1}, we have that  
\[
\begin{split}
L_\omega(m_\omega) &=L_\omega (v_\omega)+L_\omega (\chi_\omega)-L_\omega (\chi_{\sigma \omega}\circ T_\omega)\\
&=L_\omega (v_\omega)+L_\omega (\chi_\omega)-\chi_{\sigma \omega}\\
&=L_\omega (v_\omega)+\sum_{n=1}^\infty L_{\sigma^{-n} \omega}^{(n+1)}(v_{\sigma^{-n} \omega})-\sum_{n=1}^\infty L_{\sigma^{-(n-1)} \omega}^{(n)}(v_{\sigma^{-(n-1)} \omega})\\
&=L_\omega (v_\omega)+\sum_{n=2}^\infty L_{\sigma^{-(n-1)} \omega}^{(n)}(v_{\sigma^{-(n-1)} \omega})-\sum_{n=1}^\infty L_{\sigma^{-(n-1)} \omega}^{(n)}(v_{\sigma^{-(n-1)} \omega})\\
&=0,
\end{split}
\]
for $\mathbb P$-a.e. $\omega \in \Omega$.

\end{proof}

\begin{lemma}\label{MD}
Suppose that $\chi_\omega \in L^\infty(\mu_\omega)$ for $\mathbb P$-a.e. $\omega \in \Omega$. Then, 
for $\mathbb P$-a.e. $\omega \in \Omega$ and $n\in \N$, we have that 
\begin{equation}\label{904}
\mathbb E_\omega [m_{\sigma^n \omega}\circ T_\omega^{(n)}\rvert ( T_\omega^{(n+1)})^{-1}(\mathcal B)]=0,
\end{equation}
where $\mathbb E_\omega[\psi\rvert \mathcal G]$ denotes the conditional expectation of $\psi$ with respect to the $\sigma$-algebra $\mathcal G$ and  measure $\mu_\omega$. 
\end{lemma}

\begin{proof}
Using~\cite[Lemma 6]{DFGTV1}, we obtain that 
\[
\mathbb E_\omega [m_{\sigma^n \omega}\circ T_\omega^{(n)}\rvert ( T_\omega^{(n+1)})^{-1}(\mathcal B)]=L_{\sigma^n \omega}(m_{\sigma^n \omega})\circ T_\omega^{(n+1)},
\]
which in the view of~\eqref{lm} yields~\eqref{904}.
\end{proof}

\begin{remark}
Lemma \ref{MD} says that for $\mathbb P$-a.e. $\om\in \Omega$,  $(m_{\sigma^n \om}\circ T_\om^{(n)})_{n\in \N}$ is a reverse martingale difference 
with respect to the reverse filtration $(\mathcal T_\om^n)_{n\in \N}$, where $\mathcal T_\om^n=(T_\om^{(n)})^{-1}(\mathcal B)$ for $n\in \N$.
\end{remark}

We now formulate conditions  which in particular imply that $\chi_\omega$ given by~\eqref{chi} is well-defined for $\mathbb P$-a.e. $\omega \in \Omega$.
\begin{lemma}\label{chi bound}
Let $p, s, r>0$ be such that $\frac{1}{s}-\frac 1 r-\frac 1 p=\frac1 {q_0}$. Suppose that 
 $K\in L^{r}(\Omega,\mathcal F,\mathbb P)$ and $a\in L^{p}(\Omega,\mathcal F,\mathbb P)$, where $a(\omega)=\|v_\omega\|_{\mathcal H}$, $\omega \in \Omega$. Then, the random variable $\omega \mapsto
  \|\chi_\om\|_{L^\infty(\mu_\omega)}$ belongs to $L^{s}(\Omega,\mathcal F,\mathbb P)$. Moreover, for every $k\in\N$ we have 
  \begin{equation}\label{chi approx}
 \left\|\left\|\chi_\om-\sum_{j=1}^kL_{\sigma^{-j}\omega}^{(j)}(v_{\sigma^{-j} \omega})\right\|_{L^\infty(\mu_\omega)}\right\|_{L^s(\Omega,\mathcal F,\mathbb P)}\leq  \|K\|_{L^{r}(\Omega, \mathcal F, \mathbb P)}\|a\|_{L^{p}(\Omega, \mathcal F, \mathbb P)}\sum_{j=k+1}^{\infty}\|A_j\|_{L^{q_0}(\Omega, \mathcal F, \mathbb P)}. 
  \end{equation}
\end{lemma}

\begin{proof}
By \eqref{NewExpConv}, \eqref{centering} and~\eqref{chi}, we have that 
$$
\|\chi_\om\|_{L^\infty(\mu_\om)} \leq \sum_{j=1}^{\infty}K(\sigma^{-j}\om)A_j(\sigma^{-j}\omega)\|v_{\sigma^{-j}\om}\|_{\mathcal H},
$$
for $\mathbb P$-a.e. $\omega \in \Omega$. 
Thus, by the H\"older inequality  and since  $\sigma$ preserves $\mathbb P$,  we have that 
$$
\left\|\|\chi_{\om}\|_{L^\infty(\mu_\omega)}\right\|_{L^{s}(\Omega, \mathcal F, \mathbb P)}\leq \|K\|_{L^{r}(\Omega, \mathcal F, \mathbb P)}\cdot \|a\|_{L^{p}(\Omega, \mathcal F, \mathbb P)}\sum_{j=1}^{\infty}\|A_j(\omega)\|_{L^{q_0}(\Omega, \mathcal F, \mathbb P)}.
$$
It follows from~\eqref{rho norm} that $\omega \mapsto \|\chi_\omega \|_{L^\infty(\mu_\omega)}$ belongs to $L^s(\Omega, \mathcal F, \mathbb P)$.
The proof of \eqref{chi approx} is analogous.
\end{proof}

\begin{cor}\label{cor m}
Suppose that the assumptions of Lemma \ref{chi bound} hold with $s\geq 2$. Then,  all the conclusions of Theorem \ref{CLT} hold true for $u=m$ if they hold  for  $u=v$.
\end{cor}
\begin{proof}
By Lemma~\ref{chi bound} we have that 
 $\om\to \|\chi_\om\|_{L^\infty(\mu_\om)}$ belongs to $L^2(\Omega,\mathcal F,\mathbb P)$.  Hence,  Birkhoff's ergodic theorem implies that  $\|\chi_{\sigma^n\omega}\|_{L^\infty(\mu_{\sigma^n\omega})}=o(n^{1/2})$ for $\mathbb P$-a.e. $\omega \in \Omega$. Thus (see~\eqref{mdec}), 
 $$
\|S_n^\om v-S_n^\om m\|_{L^\infty(\mu_\om)}=\|\chi_\omega-\chi_{\sigma^n \omega}\circ T_\omega^{(n)}\|_{L^\infty(\mu_\om)}=
o(n^{1/2}), \quad  \text{$\mathbb P$-a.s.,}
 $$
 and the corollary follows.
\end{proof}

\begin{cor}\label{Skew chi cor}
Under the assumptions of Lemma~\ref{chi bound} with $s\ge 2$,
 for all $1\leq \gamma,\beta\leq e$ we have    that 
 $$
\lim_{n\to\infty}\int_{\Omega\times M} m^\beta \cdot (\chi^\gamma\circ\tau^n)\, d\mu=0.
 $$
\end{cor}
\begin{proof}
    Firstly, we note that  by \eqref{mdec} and Lemma \ref{Skew DEC}, it is enough to prove that 
     $$
\lim_{n\to\infty}\int_{\Omega\times M} \chi^\beta \cdot (\chi^\gamma\circ\tau^n )\, d\mu=0.
 $$
 Let us define 
 $$
E_{k}(\omega,\cdot)=\sum_{j=1}^kL_{\sigma^{-j}\om}^{(j)}(v_{\sigma^{-j}\omega}).
 $$
 Then,
 $$
\int_{\Omega\times M} \chi^\beta \cdot (\chi^\gamma\circ\tau^n )\, d\mu=
 \int_{\Omega\times M} (\chi^\beta- E_k^\beta)\cdot (\chi^\gamma\circ\tau^n)\, d\mu+
  \int_{\Omega\times M}E_k^\beta\cdot (\chi^\gamma\circ\tau^n )\, d\mu=:I_1(k,n)+I_2(k,n).
 $$
 In order to estimate $I_1(k,n)$,  note that 
 $$
I_1(k,n)=\int_{\Omega}\int_M
(\chi^\beta_\om-E_k^\beta(\omega,\cdot)) \cdot (\chi^\gamma_{\sigma^n\omega}\circ T_\omega^{(n)})\, d\mu_\omega \, d\mathbb P(\omega),
 $$
 and therefore \eqref{chi approx} and the H\"older inequality imply that 
 \[
 \begin{split}
|I_1(k,n)| &\leq \left \| \|\chi_\omega^\beta-E_k^\beta(\omega, \cdot)\|_{L^\infty(\mu_\omega)}\right \|_{L^s(\Omega, \mathcal F, \mathbb P)} \cdot \left \| \|\chi_\omega^\gamma\|_{L^\infty (\mu_\omega)}\right \|_{L^s(\Omega, \mathcal F, \mathbb P)} \\
&\le C\sum_{j>k}\|A_j\|_{L^{q_0}(\Omega,\mathcal F,\mathbb P)},
\end{split}
\]
for some constant $C>0$. Therefore, it follows from~\eqref{rho norm} that $\sup_{n}|I_1(k,n)|\to 0$ as $k\to \infty$.
Thus, it is enough to show that for every fixed $k\in \N$ we have that 
$$
\lim_{n\to\infty}I_2(k,n)=0.
$$
However, using \eqref{Duality} we see that
\[
\begin{split}
I_2(k,n) &=\int_{\Omega}\int_ME_k^\beta(\omega, \cdot) \cdot (\chi_{\sigma^n \omega}^\gamma \circ T_\omega^{(n)})\, d\mu_\omega \, d\mathbb P(\omega)\\
&=\sum_{j=1}^k\int_{\Omega}\int_M v_{\sigma^{-j}\omega}^\beta \cdot (\chi_{\sigma^n \omega}^\gamma \circ T_{\sigma^{-j}\omega}^{(n+j)})\, d\mu_{\sigma^{-j}\omega}\, d\mathbb P(\omega)\\
&=\sum_{j=1}^k\int_{\Omega\times M}v^\beta(\chi^\gamma\circ \tau^{j+n})d\mu.
\end{split}
\]
By Lemma~\ref{Skew DEC} each one of the above summands converges to $0$ as $n\to \infty$, and the proof of the corollary is complete.
 \end{proof}

In the course of the proof of our main result, we will also need the following lemma.
\begin{lemma}\label{m lemma}
 Let  $a(\cdot)$ be as in the statement of Lemma~\ref{chi bound} and 
let $p,r>0$ be such that  $K\in L^{r}(\Omega,\mathcal F, \mathbb P)$,  $a\in L^p(\Omega, \mathcal F, \mathbb P)$. 
 Then, the random variable $\omega \mapsto \|m_\om\|_{L^\infty(\mu_\omega)}$ belongs to $L^{p'}(\Omega,\mathcal F, \mathbb P)$, where $p'$ is defined by $\frac 1{p'}=\frac 1r+\frac 1p+\frac 1{q_0}.$
 \end{lemma}
\begin{proof}
By \eqref{mdec}, we have that 
$$
\|m_\om\|_{L^\infty(\mu_\om)}\leq a(\omega)+\|\chi_{\om}\|_{L^\infty(\mu_\omega)}+
\|\chi_{\sigma\omega}\|_{L^\infty(\mu_{\sigma\om})}.
$$ 
Now, it follows from Lemma~\ref{chi bound} that the random variable  $\om \mapsto \|\chi_{\om}\|_{L^\infty(\mu_\omega)}$ belongs to $L^{p'}(\Omega, \mathcal F, \mathbb P)$, which yields the desired conclusion. 
\end{proof}

\section{Iterated weak invariance principle}
For a measurable map $v\colon \Omega \times M\to \R^e$, we  consider c\`{a}dl\`{a}g processes $W_{\omega, n} \in D([0, \infty), \R^e)$ and $\mathbb{W}_{\omega, n} \in D([0, \infty), \R^{e\times e})$  defined by
\begin{equation}\label{W}
W_{\omega, n}(t):=\frac{1}{\sqrt n}\sum_{j=0}^{[nt]-1}v_{\sigma^j \omega}\circ T_\omega^{(j)}
\end{equation}
and 
\begin{equation}\label{mathbW}
\mathbb W_{\omega, n}^{\beta \gamma}(t):=\frac{1}{n}\sum_{0\le i<j\le [nt]-1}v_{\sigma^i \omega}^\beta \circ T_\omega^{(i)} \cdot v_{\sigma^j \omega}^\gamma \circ T_\omega^{(j)},
\end{equation}
for $\beta, \gamma \in \{1, \ldots, e\}$. Here, $v_\omega$ denotes $v(\omega, \cdot)$ and $D([0, \infty), \R^e)$ is the Skorokhod space.

\subsection{Preliminaries}
The following result can be regarded as  a random version of~\cite[Theorem 3.1]{KM}. 
\begin{theorem}\label{T1}
Let $p>0$, $r\geq \frac{2q_0}{q_0-2}$ and $s>2$ be such that \begin{equation}\label{wer}\frac 1 s=\frac 1 r+\frac 1 p+\frac{1}{q_0},\end{equation} and suppose that $K\in L^r(\Omega, \mathcal F, \mathbb P)$. Furthermore, assume that $v\colon \Omega \times M\to \R^e$ is a measurable map satisfying the following conditions:
\begin{itemize}
\item \eqref{spaceH} and~\eqref{centering} hold for $\mathbb P$-a.e. $\omega \in \Omega$;
\item \begin{equation}\label{b4471}
\omega \mapsto \lVert v_\omega \rVert_{\mathcal H}\in L^p(\Omega, \mathcal F, \mathbb P).
\end{equation}
\end{itemize}
Let $m\colon \Omega \times M\to \R^e$ be given by~\eqref{mdec} and take
$1\le \beta, \gamma \le e$. Then, the limit 
\[
\lim_{n\to \infty}\sum_{j=1}^n \int_{\Omega \times M} (v^\beta v^\gamma \circ \tau^j -m^\beta m^\gamma \circ \tau^j)\, d\mu
\]
exists and  for $\mathbb P$-a.e. $\omega \in \Omega$, 
\[
\mathbb W_{\omega, n}^{\beta \gamma}(t)-\mathbb M_{\omega, n}^{\beta \gamma} (t) \to t\sum_{j=1}^\infty \int_{\Omega \times M} (v^\beta v^\gamma \circ \tau^j -m^\beta m^\gamma \circ \tau^j)\, d\mu \quad \text{$\mu_\omega$-a.e.,}
\]
as $n\to \infty$, uniformly on compact subsets in $[0, \infty)$. Here, $\mathbb W_{\omega, n}$ and $\mathbb M_{\omega, n}$ are given by~\eqref{mathbW} for $v$ and $m$, respectively. 
In particular, for $\mathbb P$-a.e. $\omega \in \Omega$, the weak limits of the processes
\[
\mathbb W_{\omega, n}^{\beta \gamma}(t)-t\sum_{j=1}^n \int_{\Omega \times M} v^\beta v^\gamma \circ \tau^j \, d\mu, \quad \mathbb M_{\omega, n}^{\beta \gamma}(t)-t\sum_{j=1}^n \int_{\Omega \times M} m^\beta m^\gamma \circ \tau^j\, d\mu
\]
coincide (in the sense that if one limit exists, then so does the other and they are equal).
\end{theorem}
\begin{proof}
Firstly, we observe that 
\begin{equation}\label{constants}
\frac 1 p+\frac 1 r+\frac{1}{q_0}=\frac 1 s<\frac 1 2 \quad \text{and} \quad \frac 2 p+\frac 1 r+\frac{1}{q_0}=\frac 1 p+\frac 1 s\le \frac 2 s <1.
\end{equation}
Secondly, Lemmas~\ref{chi bound} and~\ref{m lemma} together with~\eqref{b4471} give that $\omega \mapsto \|\chi_\omega \|_{L^\infty(\mu_\omega)}$ and 
$\omega \mapsto \|m_\omega\|_{L^\infty (\mu_\omega)}$ belong to $L^s(\Omega, \mathcal F, \mathbb P)$. Write $v=m+a$, where $a:=\chi \circ \tau-\chi$. We have that
\[
\begin{split}
\mathbb W_{\omega, n}^{\beta \gamma}(t)-\mathbb M_{\omega, n}^{\beta \gamma}(t) &=\frac{1}{n}\sum_{0\le i<j\le [nt]-1}v_{\sigma^i \omega}^\beta \circ T_\omega^{(i)} \cdot v_{\sigma^j \omega}^\gamma \circ T_\omega^{(j)} \\
&\phantom{=}-\frac{1}{n}\sum_{0\le i<j\le [nt]-1}m_{\sigma^i \omega}^\beta \circ T_\omega^{(i)} \cdot m_{\sigma^j \omega}^\gamma \circ T_\omega^{(j)} \\
&=\frac 1 n \sum_{0\le i<j\le [nt]-1} a_{\sigma^i \omega}^\beta \circ T_\omega^{(i)} \cdot v_{\sigma^j \omega}^\gamma \circ T_\omega^{(j)} \\
&\phantom{=}+\frac 1 n \sum_{0\le i<j\le [nt]-1} m_{\sigma^i \omega}^\beta \circ T_\omega^{(i)} \cdot v_{\sigma^j \omega}^\gamma \circ T_\omega^{(j)} \\
&\phantom{=}-\frac 1 n \sum_{0\le i<j\le [nt]-1}m_{\sigma^i \omega}^\beta \circ T_\omega^{(i)} \cdot  v_{\sigma^j \omega}^\gamma \circ T_\omega^{(j)} \\
&\phantom{=}+\frac 1 n \sum_{0\le i<j\le [nt]-1}m_{\sigma^i \omega}^\beta \circ T_\omega^{(i)} \cdot  a_{\sigma^j \omega}^\gamma \circ T_\omega^{(j)} \\
&=I_{\omega, n}(t)+II_{\omega, n}(t), 
\end{split}
\]
where 
\[
I_{\omega, n}(t)=\frac 1 n \sum_{0\le i<j\le [nt]-1} a_{\sigma^i \omega}^\beta \circ T_\omega^{(i)} \cdot v_{\sigma^j \omega}^\gamma \circ T_\omega^{(j)}
\]
and
\[
II_{\omega, n}(t)=\frac 1 n \sum_{0\le i<j\le [nt]-1}m_{\sigma^i \omega}^\beta \circ T_\omega^{(i)} \cdot  a_{\sigma^j \omega}^\gamma \circ T_\omega^{(j)}.
\]
Observe that 
\[
\begin{split}
I_{\omega, n}(t) &=\frac 1 n \sum_{j=1}^{[nt]-1}\sum_{i=0}^{j-1}(\chi_{\sigma^{i+1}\omega}^\beta \circ T_\omega^{(i+1)}-\chi_{\sigma^i \omega}^\beta \circ T_\omega^{(i)}) v_{\sigma^j \omega}^\gamma \circ T_\omega^{(j)} \\
&=\frac 1 n \sum_{j=1}^{[nt]-1}(\chi_{\sigma^j \omega}^\beta \circ T_\omega^{(j)}-\chi_\omega^\beta)v_{\sigma^j \omega}^\gamma \circ T_\omega^{(j)} \\
&=\frac 1 n \sum_{j=1}^{[nt]-1} (\chi_{\sigma^j \omega}^\beta \cdot v_{\sigma^j \omega}^\gamma)\circ T_\omega^{(j)}-\frac 1 n \chi_\omega^\beta  \sum_{j=1}^{[nt]-1}v_{\sigma^j \omega}^\gamma \circ T_\omega^{(j)}.
\end{split}
\]
Then, it follows from Birkhoff's ergodic theorem that 
\[
\frac{1}{nt}\sum_{j=1}^{[nt]-1} (\chi^\beta \cdot v^\gamma)\circ \tau^j  \to \int_{\Omega \times M}\chi^\beta \cdot v^\gamma \, d\mu \quad \text{$\mu$-a.e.},
\]
and thus for $\mathbb P$-a.e. $\omega \in \Omega$, 
\[
\frac 1 n \sum_{j=1}^{[nt]-1} (\chi_{\sigma^j \omega}^\beta \cdot v_{\sigma^j \omega}^\gamma)\circ T_\omega^{(j)} \to t\int_{\Omega \times M}\chi^\beta \cdot v^\gamma \, d\mu, \quad \text{$\mu_\omega$-a.e.}
\]
By using Birkhoff's theorem again, for $\mathbb P$-a.e. $\omega \in \Omega$, we have that 
\[
\frac 1 n \sum_{j=1}^{[nt]-1}v_{\sigma^j \omega}^\gamma \circ T_\omega^{(j)} \to t\int_{\Omega \times M} v^\gamma \, d\mu=0, \quad \text{$\mu_\omega$-a.e.}
\]
Hence, it follows from the last two equalities that for $\mathbb P$-a.e. $\omega \in \Omega$,
\begin{equation}\label{225}
I_{\omega, n}(t)\to t\int_{\Omega \times M}\chi^\beta \cdot v^\gamma \, d\mu, \quad \text{$\mu_\omega$-a.e.}
\end{equation}
Similarly,  we have that 
\[
\begin{split}
II_{\omega, n}(t) &=\frac 1 n \sum_{i=0}^{[nt]-2}\sum_{j=i+1}^{[nt]-1}m_{\sigma^ i \omega}^\beta \circ T_\omega^{(i)} \big (\chi_{\sigma^{{j+1}}\omega}^\gamma \circ T_\omega^{(j+1)}-\chi_{\sigma^j \omega}^\gamma\circ T_\omega^{(j)}\big ) \\
&=\frac{1}{n}\sum_{i=0}^{[nt]-2}m_{\sigma^ i \omega}^\beta \circ T_\omega^{(i)} \big (\chi_{\sigma^{[nt]} \omega}^\gamma \circ T_\omega^{([nt])}-\chi_{\sigma^{i+1}\omega}^\gamma \circ T_\omega^{(i+1)} \big ) \\
&=\chi_{\sigma^{[nt]} \omega}^\gamma \circ T_\omega^{([nt])}\frac{1}{n}\sum_{i=0}^{[nt]-2}m_{\sigma^ i \omega}^\beta \circ T_\omega^{(i)} \\
&\phantom{=}-\frac{1}{n}\sum_{i=0}^{[nt]-2}(m_{\sigma^ i \omega}^\beta \circ T_\omega^{(i)}) \cdot ( \chi_{\sigma^{i+1}\omega}^\gamma \circ T_\omega^{(i+1)}).
\end{split}
\]
It follows from Birkhoff's ergodic theorem that for $\mathbb P$-a.e. $\omega \in \Omega$, 
\[
\frac{1}{n}\sum_{i=0}^{[nt]-2}(m_{\sigma^ i \omega}^\beta \circ T_\omega^{(i)}) \cdot (\chi_{\sigma^{i+1}\omega}^\gamma \circ T_\omega^{(i+1)} )\to t \int_{\Omega \times M}m^\beta \chi^\gamma \circ \tau \, d\mu, \quad \text{$\mu_\omega$-a.e.}
\]
We now claim that for $\mathbb P$-a.e. $\omega \in \Omega$, 
\begin{equation}\label{4-0-9}
\chi_{\sigma^{[nt]} \omega}^\gamma \circ T_\omega^{([nt])}\frac{1}{n}\sum_{i=0}^{[nt]-2}m_{\sigma^ i \omega}^\beta \circ T_\omega^{(i)} \to 0, \quad \text{$\mu_\omega$-a.e.}
\end{equation}
 Since $\omega \mapsto \| \chi_\omega \|_{L^\infty(\mu_\omega)}\in L^s(\Omega, \mathcal F, \mathbb P)$,  \eqref{4-0-9} follows directly from Birkhoff's ergodic theorem in the case when $s=\infty$. Next, we consider the case $s<\infty$.
By Theorem \ref{CLT}(iii) (which can be applied due to~\eqref{constants}), for  $\mathbb P$-a.e. $\omega \in \Omega$ the process $( v_{\sigma^n\omega}^\beta \circ T_\omega^{(n)})_{n\in \N}$ satisfies the law of iterated logarithm. By Corollary~\ref{cor m}, we see that for  $\mathbb P$-a.e. $\omega \in \Omega$ the process $(m_{\sigma^n\omega}^\beta \circ T_\omega^{(n)})_{n\in \N}$ satisfies the law of iterated logarithm. Thus, for $\mathbb P$-a.e. $\omega \in \Omega$,
$$
\sum_{i=0}^{n-1}m_{\sigma^ i \omega}^\beta \circ T_\omega^{(i)}=O(n^{1/2}\sqrt {\ln \ln n}),\,\,\mu_\om\text{-a.e.}
$$
Since $\omega \mapsto \| \chi_\omega \|_{L^\infty (\mu_\omega)}\in L^s(\Omega, \mathcal F, \mathbb P)$, it follows from Birkhoff's ergodic theorem that $
\|\chi_{\sigma^{n} \omega}^\gamma\|_{L^\infty(\mu_{\sigma^n\om})}=o(n^{1/s})$ for $\mathbb P$-a.e. $\omega \in \Omega$. 
Indeed,   let $\Psi(\omega):=\|\chi_\omega^\gamma\|_{L^\infty(\mu_\omega)}^s$ for $\omega \in \Omega$. Then, $\Psi \in L^1(\Omega, \mathcal F, \mathbb P)$. By Birkhoff's ergodic theorem, 
\[
\lim_{n\to \infty}\frac{1}{n}\sum_{i=0}^{n-1}\Psi(\sigma^i \omega)=\int_\Omega \Psi \, d\mathbb P \quad \text{for $\mathbb P$-a.e. $\omega \in \Omega$, }
\]
which gives that $\lim\limits_{n\to \infty}\frac{1}{n}\Psi(\sigma^n \omega)=0$ for $\mathbb P$-a.e. $\omega \in \Omega$. Hence, 
$\lim\limits_{n\to \infty}\frac{\|\chi_{\sigma^n\omega}^\gamma\|_{L^\infty(\mu_{\sigma^n \omega})}}{n^{1/s}}=0$ for $\mathbb P$-a.e. $\omega \in \Omega$, yielding the desired conclusion.
Since $s>2$ we conclude that~\eqref{4-0-9} holds in this case as well.

Therefore, for $\mathbb P$-a.e. $\omega \in \Omega$, we have that 
\begin{equation}\label{226}
II_{\omega, n}(t) \to -t \int_{\Omega \times M}m^\beta \chi^\gamma \circ \tau \, d\mu, \quad \text{$\mu_\omega$-a.e.}
\end{equation}
By~\eqref{225} and~\eqref{226}, we conclude that for $\mathbb P$-a.e. $\omega \in \Omega$,
\begin{equation}\label{356}
\mathbb W_{\omega, n}^{\beta \gamma}(t)-\mathbb M_{\omega, n}^{\beta \gamma}(t)  \to t\bigg{(}\int_{\Omega \times M}\chi^\beta v^\gamma \, d\mu -\int_{\Omega \times M}m^\beta \chi^\gamma \circ \tau \, d\mu\bigg{)}, \ \text{$\mu_\omega$-a.e.}
\end{equation}
On the other hand, we have that 
\[
v^\beta \cdot(v^\gamma \circ \tau^j) -m^\beta \cdot(m^\gamma \circ \tau^j)=(\chi^\beta \circ \tau-\chi^\beta)v^\gamma \circ \tau^j+m^\beta (\chi^\gamma \circ \tau-\chi^\gamma)\circ \tau^j.
\]
Hence, using that $\tau$ preserves $\mu$ we have that
\[
\begin{split}
& \sum_{j=1}^n \int_{\Omega \times M}v^\beta\cdot(v^\gamma \circ \tau^j) \, d\mu -\sum_{j=1}^n \int_{\Omega \times M} m^\beta \cdot(m^\gamma \circ \tau^j)\, d\mu \\
&=\sum_{j=1}^n \int_{\Omega \times M}\bigg{\{}(\chi^\beta \circ \tau-\chi^\beta)v^\gamma \circ \tau^j+m^\beta (\chi^\gamma \circ \tau-\chi^\gamma)\circ \tau^j \bigg{\}}\, d\mu \\
&=\sum_{j=1}^n \int_{\Omega \times M}\bigg{\{} (\chi^\beta \circ \tau^{n-j+1} -\chi^\beta \circ \tau^{n-j})v^\gamma \circ \tau^n \\
&\phantom{=\sum_{j=1}^n \int_{\Omega \times M}\bigg{\{}}+m^\beta (\chi^\gamma \circ \tau^{j+1}-\chi^\gamma \circ \tau^j ) \bigg{\}}\, d\mu \\
&=\int_{\Omega \times M}\chi^\beta v^\gamma \, d\mu-\int_{\Omega \times M}m^\beta \cdot(\chi^\gamma \circ \tau) \, d\mu) +L_n, 
\end{split}
\]
where 
\[
L_n:=\int_{\Omega \times M}(m^\beta\cdot(\chi^\gamma \circ \tau^{n+1})-\chi^\beta\cdot(v^\gamma \circ \tau^n))\, d\mu.
\]
Now,  $L_n \to 0$ as $n\to \infty$ by Corollary \ref{Skew chi cor}.
The conclusion of the theorem now follows directly from~\eqref{356}.
\end{proof}

$$
$$
As a direct consequence of the previous theorem, we obtain the following corollary. 
\begin{cor}\label{co}
Let the assumptions of Theorem~\ref{T1} hold, $\hat W\in D([0, \infty), \mathbb R^e)$ and $\hat{\mathbb W}\in D([0, \infty), \mathbb R^{e\times e})$. 
Furthermore,  suppose that for $\mathbb P$-a.e. $\omega \in \Omega$, $(M_{\omega, n}, \mathbb M_{\omega, n})\to_w (\hat W, \hat{\mathbb W})$ in $D([0, \infty), \R^e \times \R^{e\times e})$, where $M_{\omega, n}$ and $\mathbb M_{\omega, n}$ are given by~\eqref{W} and~\eqref{mathbW} by replacing $v$ by $m$.
Then, for $\mathbb P$-a.e. $\omega \in \Omega$,  $(W_{\omega, n}, \mathbb W_{\omega, n})\to_w (W, \mathbb W)$ in $D([0, \infty), \R^e \times \R^{e\times e})$, where $W=\hat W$ and 
\[
\mathbb W^{\beta \gamma}(t)=\mathbb M^{\beta \gamma} (t)+t\sum_{j=1}^\infty \int_{\Omega \times M} (v^\beta v^\gamma \circ \tau^j -m^\beta m^\gamma \circ \tau^j)\, d\mu.
\]
\end{cor}

\subsection{Iterated weak invariance principle for martingales}
Throughout this subsection we consider a measurable $m\colon \Omega \times M\to \R^e$ with the property that $m_\omega=m(\omega, \cdot)\in L^\infty(\mu_\omega)$ for $\mathbb P$-a.e. $\omega \in \Omega$. Moreover, we require that  \eqref{lm} holds for  $\mathbb P$-a.e. $\omega \in \Omega$.
 Let $M_{\omega, n}\in D([0, \infty), \R^e)$ and $\mathbb {M}_{\omega, n}\in D([0, \infty), \R^{e\times e})$ be defined as~\eqref{W} and~\eqref{mathbW}, by replacing $v$ with $m$.

In order to establish the weak invariance principle for $(M_{\omega, n}, \mathbb M_{\omega, n})$, we need several auxiliary results. For $\omega \in \Omega$, set 
\[
\hat{M}_\omega:=\{\mathbf x=(x_n)_{n\in \Z}\subset M: T_{\sigma^n \omega}(x_n)=x_{n+1}, \  \forall n\in \Z\}.
\]
Moreover, let $\hat{T}_\omega \colon \hat{M}_\omega \to \hat{M}_{\sigma \omega}$ be given by
\[
(\hat{T}_\omega (\mathbf x))_n=x_{n+1}=T_{\sigma^n \omega}(x_n), \quad \mathbf x=(x_n)_{n\in \Z}\in \hat{M}_\omega, \ n\in \Z.
\]
Observe that $\hat T_\omega$ is an invertible transformation and that its inverse is given by 
\[
((\hat{T}_\omega)^{-1} (\mathbf x))_n=x_{n-1},  \quad \mathbf x=(x_n)_{n\in \Z}\in \hat{M}_{\sigma \omega}, \ n\in \Z.
\]
For $j\in \N$ and $\omega \in \Omega$, set
\[
\hat T_\omega^{(j)}=\hat{T}_{\sigma^{j-1} \omega} \circ \ldots \circ \hat T_\omega \quad \text{and} \quad \hat T_\omega^{(-j)}=(\hat T_{\sigma^{-j} \omega}^{(j)})^{-1}.
\]
In addition, we consider  canonical projections $i_\omega \colon \hat{M}_\omega \to M$ defined by 
\[
i_\omega(\mathbf x)=x_0, \quad \mathbf x=(x_n)_{n\in \Z}\in \hat{M}_\omega. 
\]
One can easily verify that
\begin{equation}\label{iT}
i_{\sigma \omega}\circ \hat T_\omega=T_\omega \circ i_\omega, \quad \text{for $\omega \in \Omega$.}
\end{equation}
For $\omega \in \Omega$, set \[\hat{\mathcal B}_\omega=\{i_\omega^{-1}(B): B\in \mathcal B\}.\]
Then, $\hat{\mathcal B}_\omega$ is a $\sigma$-algebra on $\hat{M}_\omega$.
For $\omega \in \Omega$, we define the measure $\hat{\mu}_\omega$ on $(\hat{M}_\omega, \hat{\mathcal B}_\omega)$ by 
\[
\hat{\mu}_\omega (i_\omega^{-1}(B))=\mu_\omega(B), \quad  \text{for $B\in \mathcal B$.}
\]
We make the following simple observation.
\begin{lemma}
For $\mathbb P$-a.e. $\omega \in \Omega$,
\[
(\hat{T}_\omega)^*\hat{\mu}_\omega=\hat{\mu}_{\sigma \omega}.
\]
\end{lemma}
\begin{proof}
It follows from~\eqref{equi} and~\eqref{iT} that 
\[
\begin{split}
(\hat{T}_\omega)^*\hat{\mu}_\omega (i_{\sigma \omega}^{-1}(B)) &=\hat{\mu}_\omega ((i_{\sigma \omega}\circ \hat{T}_\omega)^{-1}(B)) \\
&=\hat{\mu}_\omega (i_\omega^{-1}(T_\omega^{-1}(B))) \\
&=\mu_\omega (T_\omega^{-1}(B))) \\
&=T_\omega^*\mu_\omega (B)\\
&=\mu_{\sigma \omega}(B)\\
&=\hat{\mu}_{\sigma \omega}(i_{\sigma \omega}^{-1}(B)),
\end{split}
\]
for $\mathbb P$-a.e. $\omega \in \Omega$ and $B\in \mathcal B$. The proof of the lemma is completed. 
\end{proof}
Similarly to~\eqref{tau}, we consider the skew-product transformation $\hat{\tau}$ given by
\[
\hat{\tau}(\omega, \mathbf x)=(\sigma \omega, \hat{T}_\omega (\mathbf x)), \quad \text{for $\omega \in \Omega$ and $\mathbf x\in \hat{M}_\omega$.}
\]
Observe that $\hat{\tau}$ is invertible and that its inverse is given by
\[
(\hat{\tau})^{-1}(\omega, \mathbf x)=(\sigma^{-1} \omega, \hat{T}_\omega^{(-1)}(\mathbf x)), \quad \text{for $\omega \in \Omega$ and $\mathbf x\in \hat{M}_\omega$.}
\]
Observe that $\hat{\tau}$ (and hence also its inverse) preserve the measure $\hat{\mu}$ given by 
\[
\hat{\mu}(C)=\int_\Omega  \hat{\mu}_\omega (C_\omega)\, d\mathbb P(\omega), 
\]
where $C\subset
\{(\omega, \mathbf x): \omega \in \Omega, \ \mathbf x\in \hat{M}_\omega \}$ is measurable and $C_\omega=\{\mathbf x\in \hat{M}_\omega: (\omega, \mathbf x)\in C\}$. In addition, 
the ergodicity of $\mu$ implies that $\hat{\mu}$ is also ergodic. In what  follows, $\hat{\mathbb E}_\omega [\psi \rvert \mathcal G]$ will denote the conditional expectation of $\psi$ with respect to the measure $\hat{\mu}_\omega$  and a $\sigma$-algebra $\mathcal G$. For $\omega \in \Omega$, set 
\[
\tilde m_\omega :=m_\omega \circ i_\omega. 
\]
\begin{lemma}\label{martingale}
For $\mathbb P$-a.e. $\omega \in \Omega$ and $n\in \N$, we have that
\[
\hat{\mathbb E}_\omega [\tilde m_{\sigma^{-n} \omega}\circ \hat T_\omega^{(-n)} \rvert \hat T_{\sigma^{-(n-1)}\omega}^{(n-1)}\hat {\mathcal B}_{\sigma^{-(n-1)}\omega}]=0.
\]
\end{lemma}
\begin{proof}
For $B\in \mathcal B$, writing  $A=(i_{\sigma^{-(n-1)} \omega})^{-1}(B)$ we have that
\[
\begin{split}
&\int_{\hat T_{\sigma^{-(n-1)}\omega}^{(n-1)}(A) } \tilde m_{\sigma^{-n} \omega}\circ \hat T_\omega^{(-n)}\, d \hat{\mu}_\omega  \\
&=\int_{\hat M_\omega}\tilde m_{\sigma^{-n} \omega}\circ \hat T_\omega^{(-n)} \cdot \textbf{1}_A \circ (\hat T_{\sigma^{-(n-1)}\omega}^{(n-1)})^{-1}\, d\hat \mu_\omega \\
&=\int_{\hat M_{\sigma^{-n} \omega}}\tilde m_{\sigma^{-n} \omega} \cdot (\textbf{1}_A \circ \hat T_{\sigma^{-n} \omega})\,  d\hat \mu_{\sigma^{-n} \omega}\\
&=\int_{\hat M_{\sigma^{-n} \omega}} (m_{\sigma^{-n} \omega}  \circ i_{\sigma^{-n} \omega})\cdot (\textbf{1}_B \circ i_{\sigma^{-(n-1)}\omega} \circ \hat T_{\sigma^{-n} \omega})\,  d\hat \mu_{\sigma^{-n} \omega}\\
&=\int_{\hat M_{\sigma^{-n} \omega}} (m_{\sigma^{-n} \omega}  \circ i_{\sigma^{-n} \omega})\cdot (\textbf{1}_B \circ  T_{\sigma^{-n} \omega} \circ i_{\sigma^{-n} \omega})\,  d\hat \mu_{\sigma^{-n} \omega}\\
&=\int_M m_{\sigma^{-n} \omega} \cdot (\textbf{1}_B \circ  T_{\sigma^{-n} \omega}) \, d\mu_{\sigma^{-n} \omega}\\
&=\int_{(T_{\sigma^{-n}\omega})^{-1}(B)}m_{\sigma^{-n} \omega} \, d\mu_{\sigma^{-n} \omega}\\
&=\int_{(T_{\sigma^{-n}\omega})^{-1}(B)} \mathbb E_{\sigma^{-n} \omega} [m_{\sigma^{-n} \omega} \rvert (T_{\sigma^{-n}\omega})^{-1}(\mathcal B)]  \, d\mu_{\sigma^{-n} \omega} \\
&=0,
\end{split}
\]
where in the last step we used~\eqref{904}.
The proof of the lemma is completed. 
\end{proof}
For $\omega \in \Omega$ and $n\in \N$, set
\[
\tilde M_{\omega, n}^{-}(t):=\frac{1}{\sqrt n} \sum_{j=-[nt]}^{-1} \tilde m_{\sigma^j \omega}\circ \hat  T_\omega^{(j)}, \quad t\ge 0.
\]
Furthermore, for $1\le \beta, \gamma \le e$, we define
\[
\tilde{\mathbb M}_{\omega, n}^{\beta \gamma, -}(t):=\frac 1 n \sum_{-[nt] \le j <i \le -1} \tilde m_{\sigma^i \omega}^\beta \circ \hat T_\omega^{(i)} \cdot \tilde m_{\sigma^j \omega}^\gamma \circ \hat T_\omega^{(j)}, \quad t\ge 0.
\]
\begin{lemma}\label{tl}
Suppose $\omega \mapsto \|m_\om\|_{L^\infty(\mu_\omega)}\in L^3(\Omega,\mathcal F,\mathbb P)$. Then 
for $\mathbb P$-a.e. $\omega \in \Omega$, we have that $\tilde M_{\omega, n}^- \to_w W$ in $D([0, \infty), \R^e)$, where $W$ is the $e$-dimensional Brownian motion with the covariance matrix $\Cov (W(1))=\int_{\Omega \times M}mm^T \, d\mu$.
\end{lemma}

\begin{proof}
Fix $1\le \beta, \gamma \le e$ and 
set
\[
V_{\omega, n}^{\beta\gamma}:=\sum_{j=-n}^{-1}\hat{\mathbb E}_\omega  [(\tilde m_{\sigma^j \omega}^\beta  \tilde m_{\sigma^j \omega}^\gamma)\circ  \hat{T}_\omega^{(j)}\rvert \hat{T}_{\sigma^{j+1}\omega}^{(-j-1)}\hat{\mathcal B}_{\sigma^{j+1}\omega}].
\]
Note that
\[
\hat{\mathbb E}_\omega  [(\tilde m_{\sigma^j \omega}^\beta  \tilde m_{\sigma^j \omega}^\gamma)\circ  \hat{T}_\omega^{(j)}\rvert \hat{T}_{\sigma^{j+1}\omega}^{(-j-1)}\hat{\mathcal B}_{\sigma^{j+1}\omega}]=
\hat{\mathbb E}_{\sigma^j \omega}[\tilde m_{\sigma^j \omega}^\beta  \tilde m_{\sigma^j \omega}^\gamma\rvert (\hat{T}_{\sigma^j \omega})^{-1}\hat{\mathcal B}_{\sigma^{j+1}\omega}]\circ \hat{T}_\omega^{(j)},
\]
and thus
\[
V_{\omega, n}^{\beta \gamma}=\sum_{j=-n}^{-1}\hat{\mathbb E}_{\sigma^j \omega}[\tilde m_{\sigma^j \omega}^\beta  \tilde m_{\sigma^j \omega}^\gamma\rvert (\hat{T}_{\sigma^j \omega})^{-1}\hat{\mathcal B}_{\sigma^{j+1}\omega}]\circ \hat{T}_\omega^{(j)}.
\]
We define an observable $\psi$ by
\[
\psi (\omega,  \mathbf x)=\hat{\mathbb E}_{ \omega}[\tilde m_\omega^\beta \tilde m_\omega^\gamma \rvert (\hat T_\omega)^{-1}\hat{\mathcal B}_{\sigma \omega}](\mathbf x), \quad \text{for $\omega \in \Omega$ and $\mathbf x\in \hat M_\omega$.}
\]
By applying Birkhoff's ergodic theorem for $\psi$, $(\hat{\tau})^{-1}$ and $\hat \mu$, we conclude that for $\mathbb P$-a.e. $\omega \in \Omega$, 
\begin{equation}\label{959}
\frac{V_{\omega, n}^{\beta \gamma}}{n} \to \int  \tilde m^\beta  \tilde m^\gamma \, d{\hat \mu}=\int m^\beta m^\gamma \, d\mu, \quad \text{$\hat \mu_\omega$-a.e.}
\end{equation}
Let us now set 
\[
\sigma_{\omega, n}^{\beta \gamma}:=\sum_{j=-n}^{-1} \int ( \tilde m_{\sigma^j \omega}^\beta \tilde m_{\sigma^j \omega}^\gamma) \circ \hat    T_\omega^{(j)}\, d \hat \mu_\omega=\sum_{j=-n}^{-1} \int  \tilde m_{\sigma^j \omega}^\beta  \tilde m_{\sigma^j \omega}^\gamma \, d \hat \mu_{\sigma^j \omega}.
\]
By Birkhoff's ergodic theorem, we have that for $\mathbb P$-a.e. $\omega \in \Omega$, 
\begin{equation}\label{960}
\frac{\sigma_{\omega, n}^{\beta \gamma}}{n} \to \int \tilde m^\beta \tilde m^\gamma \, d\hat \mu=\int  m^\beta m^\gamma\, d\mu.
\end{equation}

On the other hand, using H\"{o}lder's and Markov's inequalities, for an arbitrary $\epsilon >0$ and  $\mathbb P$-a.e. $\omega \in \Omega$ we see that, 
\[
\begin{split}
&\frac{1}{ n}\sum_{j=-n}^{-1}\int_{\hat M_\omega}( \tilde m_{\sigma^j \omega}^\beta  )^2\circ   \hat T_\omega^{(j)}\textbf{1}_{ \{\lvert (\tilde  m_{\sigma^j \omega}^\beta )^2\circ  \hat T_\omega^{(j)} \rvert \ge \epsilon \sqrt{n} \}}\, d\hat \mu_\omega \\
&\le \frac{1}{ n}\sum_{j=-n}^{-1}\lVert (\tilde  m_{\sigma^j \omega}^\beta  )^2 \circ   \hat T_\omega^{(j)}\rVert_{L^2(\hat \mu_\omega)} \cdot (\hat \mu_\omega \{\lvert ( \tilde m_{\sigma^j \omega}^\beta )^2\circ  \hat T_\omega^{(j)} \rvert \ge \epsilon \sqrt n\})^{1/2}\\
&\le \frac{1}{ n}\sum_{j=-n}^{-1}\lVert ( \tilde m_{\sigma^j \omega}^\beta )^2\circ \hat T_\omega^{(j)}\rVert_{L^2(\hat \mu_\omega)}\cdot \left (\frac{1}{\epsilon \sqrt n} \|(\tilde  m_{\sigma^j \omega}^\beta)^2\circ   \hat T_\omega^{(j)}\rVert_{L^1(\hat \mu_\omega)}\right )^{1/2}\\
&=\frac{1}{\sqrt{\epsilon} n^{5/4}}\sum_{j=-n}^{-1}\| (\tilde m_{\sigma^j \omega}^\beta)^2\|_{L^2(\hat \mu_{\sigma^j \omega})} \cdot \| (\tilde m_{\sigma^j \omega}^\beta)^2\rVert_{L^1(\hat \mu_{\sigma^j \omega})}^{1/2}\\
&\le \frac{1}{\sqrt{\epsilon} n^{5/4}}\sum_{j=-n}^{-1}\| (\tilde m_{\sigma^j \omega}^\beta )^2\|_{L^2(\hat \mu_{\sigma^j \omega})}^{3/2}\\
&\le  \frac{1}{\sqrt{\epsilon} n^{5/4}} \sum_{j=-n}^{-1}  \|m_{\sigma^j \omega}\|_{L^\infty(\mu_{\sigma^j\omega})}^3.
\end{split}
\]
Now, using Birkhoff's ergodic theorem we have $\sum_{j=-n}^{-1}\|m_{\sigma^j \omega}\|_{L^\infty(\mu_{\sigma^j\omega})}^3=O(n)$ 
and thus,
\begin{equation}\label{c2}
 \frac{1}{ n}\sum_{j=-n}^{-1}\int_{\hat M_\omega}( \tilde m_{\sigma^j \omega}^\beta  )^2\circ   \hat T_\omega^{(j)}\textbf{1}_{ \{\lvert (\tilde  m_{\sigma^j \omega}^\beta )^2\circ  \hat T_\omega^{(j)} \rvert \ge \epsilon \sqrt n \}}\, d\hat \mu_\omega\to 0,
\end{equation}
for $\mathbb P$-a.e. $\omega \in \Omega$. 
 The conclusion of the lemma follows from~\eqref{959}, \eqref{960}, \eqref{c2} and~\cite[Theorem 2]{Brown}.
\end{proof}

\begin{lemma}\label{l143}
Suppose $\omega \mapsto \|m_\om\|_{L^\infty(\mu_\om)}\in L^3(\Omega,\mathcal F,\mathbb P)$. Then 
for $\mathbb P$-a.e. $\omega \in \Omega$, we have that $(\tilde M_{\omega, n}^{-}, \tilde{ \mathbb{M}}_{\omega, n}^{-})\to_w (W, J)$ in $D([0, \infty), \R^e \times \R^{e\times e})$,
where $W$ is as in the statement of Lemma~\ref{tl} and $J^{\beta \gamma}(t)=\int_0^t W^\beta \, d W^\gamma$.
\end{lemma}

\begin{proof}
For $t\ge 0$ and $1\le \gamma \le e$, we have that
\[
\begin{split}
\int (\tilde M_{\omega, n}^{\gamma, -}(t))^2 \, d\hat{\mu}_\omega &=\frac 1 n \int \bigg ( \sum_{j=-[nt]}^{-1}\tilde m_{\sigma^j \omega}^\gamma \circ \hat{T}_\omega^{(j)} \bigg{)}^2\, d\hat{\mu}_\omega  \\
&=\frac 1 n \sum_{j=-[nt]}^{-1}\int  (\tilde m_{\sigma^j \omega}^\gamma)^2 \circ \hat{T}_\omega^{(j)}\, d\hat{\mu}_\omega \\
&=\frac{1}{n}\sum_{j=-[nt]}^{-1}\int  (\tilde m_{\sigma^j \omega}^\gamma)^2 \, d\hat{\mu}_{\sigma^j \omega},
\end{split}
\]
and thus it follows from Birkhoff's ergodic theorem (as in~\eqref{960}) that 
\[
\int (\tilde M_{\omega, n}^{\gamma, -}(t))^2 \, d\hat{\mu}_\omega \to t\int (\tilde m^\gamma )^2\, d\hat{\mu},
\]
for $\mathbb P$-a.e. $\omega \in \Omega$.  In particular, for $\mathbb P$-a.e. $\omega \in \Omega$, $\sup_{n}\int (\tilde M_{\omega, n}^{\gamma, -}(t))^2 \, d\hat{\mu}_\omega<+\infty$. 
The conclusion of the lemma now follows from Lemma~\ref{martingale} (which ensures that $\tilde M^-_{\omega, n}$ is a martingale) and~\cite[Theorem 2.2 and Remark 2.3]{KP} applied for $\delta=\infty$ and $A_n^{\delta}=0$. We note that above we verified that condition~\cite[(C2.2)(iii)]{KP} is in force.
\end{proof}
We now wish to relate the convergence of $(\tilde M_{\omega, n}^{-}, \tilde{ \mathbb{M}}_{\omega, n}^{-})$ and $(\tilde M_{\omega, n}, \tilde{ \mathbb{M}}_{\omega, n})$, where
$\tilde M_{\omega, n}(t)$ and $\tilde{\mathbb M}_{\omega, n}^{\beta \gamma}(t)$ are given by~\eqref{W} and~\eqref{mathbW} by replacing $v$ with $\tilde m$.
The following is a version of~\cite[Lemma 4.8]{KM}. 
\begin{lemma}\label{l142}
Let $T\in \N$ and
suppose that  $\omega \mapsto \|m_\om\|_{L^\infty(\mu_\omega)}\in L^{4}(\Omega,\mathcal F,\mathbb P)$.
Let $g(u)(t)=u(T)-u(T-t)$ and $h(v)(t)=v(T-t)(v(T)-v(T-t))$, Furthermore, let $\ast$ denote matrix transpose in $\R^{e\times e}$. Then, for $\mathbb P$-a.e. $\omega \in \Omega$ and every $n\in \N$ we have that
\[
\sup_{t\in[0,T]}\left|(\tilde M_{\omega, n}, \tilde{\mathbb M}_{\omega, n})(t)\circ \hat T_{\sigma^{nT}\om}^{(-nT)}-\bigg{(}g( \tilde M^{-}_{\sigma^{nT}\omega, n}), \big ( g(\tilde{\mathbb M}_{\sigma^{nT}\omega, n}^- )-h (\tilde M_{\sigma^{nT} \omega, n}^-) \big )^*(t)\right|
\]
$$
\leq F_{\om,n}\circ  \hat T_{\sigma^{nT}\om}^{(-nT)},
$$
where ($\mathbb P$-a.s.)
$$
\lim_{n\to\infty}F_{\om,n}=0, \,\,\hat\mu_\om\text{-a.s.}
$$
\end{lemma}
\begin{proof}
We have that
\begin{equation}\label{h1a}
\begin{split}
\tilde M_{\omega, n}(t)\circ \hat T_{\sigma^{nT} \omega}^{(-nT)} &=\frac{1}{\sqrt n}\sum_{j=0}^{[nt]-1}\tilde m_{\sigma^ j \omega}\circ \hat T_\omega^{(j)}\circ \hat T_{\sigma^{nT} \omega}^{(-nT)} \\
&=\frac{1}{\sqrt n}\sum_{j=0}^{[nt]-1}\tilde m_{\sigma^ j \omega} \circ \hat T_{\sigma^{nT}\omega}^{(-nT+j)}\\
&=\frac{1}{\sqrt n}\sum_{j=-nT}^{[nt]-1-nT}\tilde m_{\sigma^j (\sigma^{nT} \omega)} \circ \hat T_{\sigma^{nT} \omega}^{(j)}.
\end{split}
\end{equation}
Hence,
\begin{equation}\label{h1b}
\tilde M_{\omega, n}(t) \circ \hat T_{\sigma^{nT} \omega}^{(-nT)}=\tilde M_{\sigma^{nT}\omega, n}^-(T)-\tilde M_{\sigma^{nT}\omega, n}^-(T-t)+G_{\omega, n}^0 (t), 
\end{equation}
where $G_{\omega, n}^0 (t)$ consists of at most one term and  
\[
\sup_{t\in [0, T]}|G_{\om, n}^0(t)| \le \frac{1}{\sqrt n}\sup_{t\in [0, T]}\lvert \tilde m_{\sigma^{[nt]-1}\omega}\circ \hat T_{\sigma^{nT} \omega}^{([nt]-1-nT)}\rvert \le F_{\omega, n}^0\circ \hat T_{\sigma^{nT}\omega}^{(-nT)} \quad \text{$\hat \mu_\omega$-a.e.},
\]
with \[F_{\omega, n}^0=\frac{1}{\sqrt n}\sup_{t\in [0, T]}\|m_{\sigma^{[nt]-1}\omega}\|_{L^\infty(\mu_{\sigma^{[nt]-1}\om})}.\]
By  Birkhoff's ergodic theorem, we have that  $\frac 1 n \|m_{\sigma^{n-1} \omega}\|_{L^\infty(\mu_{\sigma^{n-1}\om})}^2 \to 0$ for $\mathbb P$-a.e. $\omega \in \Omega$.
This readily implies that $ F_{\omega, n}^0 \to 0$,\, $\hat \mu_\omega$-a.e.
On the other hand, we have that 
\[
\begin{split}
& \tilde{\mathbb M}_{\omega, n}^{\beta \gamma} (t)\circ \hat{T}_{\sigma^{nT} \omega}^{(-nT)} \\
 &= \frac 1 n \bigg{(}\sum_{j=1}^{[nt]-1}\sum_{i=0}^{j-1} \tilde m_{\sigma^i \omega}^\beta \circ \hat T_\omega^{(i)} \cdot \tilde m_{\sigma^j \omega}^\gamma \circ \hat T_\omega^{(j)} \bigg )\circ \hat{T}_{\sigma^{nT} \omega}^{(-nT)}\\
&=\frac 1 n \sum_{j=-nT+1}^{[nt]-1-nT}\sum_{i=-nT}^{j-1} \tilde m_{\sigma^i (\sigma^{nT} \omega)}^\beta \circ \hat T_{\sigma^{nT}\omega}^{(i)} \cdot \tilde m_{\sigma^j (\sigma^{nT}\omega)}^\gamma \circ \hat T_{\sigma^{nT}\omega}^{(j)} \\
&=\frac 1n \bigg (\sum_{j=-nT+1}^{-1}-\sum_{j=[nt]-nT+1}^{-1}-\sum_{j=[nt]-nT}^{[nt]-nT} \bigg )\sum_{i=-nT}^{j-1} \tilde m_{\sigma^i (\sigma^{nT} \omega)}^\beta \circ \hat T_{\sigma^{nT}\omega}^{(i)} \cdot \tilde m_{\sigma^j (\sigma^{nT}\omega)}^\gamma \circ \hat T_{\sigma^{nT}\omega}^{(j)} \\
&=\tilde{\mathbb M}_{\sigma^{nT}\omega, n}^{\gamma \beta, -}(T)-E_{\omega, n}(t)-G_{\omega, n}^2(t), 
\end{split}
\]
where 
\[
G_{\omega, n}^2(t)=\bigg{(}\frac{1}{\sqrt n}\sum_{i=-nT}^{[nt]-nT-1}\tilde m_{\sigma^i (\sigma^{nT} \omega)}^\beta \circ \hat T_{\sigma^{nT}\omega}^{(i)} \bigg )\cdot \bigg (\frac{1}{\sqrt n}\tilde m_{\sigma^{[nt]}\omega}^\gamma \circ \hat T_{\sigma^{nT}\omega}^{([nt]-nT)}\bigg ),
\]
and
\[
E_{\omega, n}(t)=\frac 1 n \sum_{j=[nt]-nT+1}^{-1} \sum_{i=-nT}^{j-1} \tilde m_{\sigma^i (\sigma^{nT} \omega)}^\beta \circ \hat T_{\sigma^{nT}\omega}^{(i)}\cdot  \tilde m_{\sigma^j (\sigma^{nT}\omega)}^\gamma \circ \hat T_{\sigma^{nT} \omega}^{(j)}.
\]
Next, note   that 
\[
\frac{1}{\sqrt n}\sum_{i=-nT}^{[nt]-nT-1}\tilde m_{\sigma^i (\sigma^{nT} \omega)}^\beta \circ \hat T_{\sigma^{nT}\omega}^{(i)}=\tilde M_{\omega, n}^\beta(t)\circ \hat T_{\sigma^{nT}\omega}^{(-nT)}.
\]
Hence,
\[
\sup_{t\in [0, T]}|G_{\omega, n}^2(t)| \le F_{\omega, n}^2 \circ T_{\sigma^{nT}\om}^{(-nT)},
\]
where 
\[
F_{\omega, n}^2=\frac{1}{\sqrt n}\sup_{t\in [0, T]}(|\tilde M_{\omega, n}^\beta(t)| \cdot \|m_{\sigma^{[nt]}\om}\|_{L^\infty (\mu_{\sigma^{[nt]}\om})}).
\]
One can show that $\tilde M_{\omega, n}(t)$ converges to $W$ weakly. Indeed, this can be proved by arguing as in the first part of the proof of the current lemma (\eqref{h1a}, \eqref{h1b} and the arguments following it),  Corollary~\ref{corr} and Lemma~\ref{411} by ignoring the second component of the process. This now easily implies that $F_{\omega, n}^2\to 0$.
Let us now estimate $E_{\omega,n}(t)$. Notice first that since $T\in\N$, either $nT-[nt]=[n(T-t)]$ or $nT-[nt]=1+[n(T-t)]$. For the sake of simplicity let us estimate first 
$E_{\omega, n}(t)$ when $nT-[nt]=[n(T-t)]$. In this case we have
\[
\begin{split}
E_{\omega, n}(t) &=\frac 1 n \sum_{j=[nt]-nT+1}^{-1} \bigg (\sum_{i=-nT}^{-nT+[nt]-1}+\sum_{i=-nT+[nt]}^{j-1} \bigg )\tilde m_{\sigma^i (\sigma^{nT} \omega)}^\beta \circ \hat T_{\sigma^{nT}\omega}^{(i)}\cdot \tilde m_{\sigma^j (\sigma^{nT}\omega)}^\gamma \circ \hat T_{\sigma^{nT} \omega}^{(j)} \\
&=H_{\omega, n}(t)+\tilde{\mathbb M}_{\sigma^{nT} \omega, n}^{\gamma \beta, -}(T-t)-F_{\omega, n}^3(t), 
\end{split}
\]
where
\[
\begin{split}
H_{\omega, n}(t) &= \bigg ( \frac{1}{\sqrt n}\sum_{j=[nt]-nT}^{-1}\tilde m_{\sigma^j (\sigma^{nT}\omega)}^\gamma \circ \hat T_{\sigma^{nT} \omega}^{(j)} \bigg ) \bigg (\frac{1}{\sqrt n}\sum_{i=-nT}^{-nT+[nt]-1}\tilde m_{\sigma^i (\sigma^{nT} \omega)}^\beta \circ \hat T_{\sigma^{nT}\omega}^{(i)} \bigg ) \\
&=\tilde M_{\sigma^{nT}\omega, n}^{\gamma, -}(T-t) \bigg (\tilde M_{\sigma^{nT} \omega, n}^{\beta, -}(T)-\tilde M_{\sigma^{nT} \omega, n}^{\beta, -}(T-t) \bigg ),
\end{split}
\]
and 
\[
\begin{split}
F_{\omega, n}^3(t) &=\frac 1 n \sum_{i=-nT}^{-nT+[nt]-1}\tilde m_{\sigma^i (\sigma^{nT} \omega)}^\beta \circ \hat T_{\sigma^{nT}\omega}^{(i)} \cdot \tilde m_{\sigma^{[nt]}\omega }^\gamma \circ \hat T_{\sigma^{nT} \omega}^{([nt]-nT)} \\
&=\frac 1 n\tilde m_{\sigma^{[nt]}\omega }^\gamma\circ \hat T_{\sigma^{nT}\omega}^{([nt]-nT)}\left(\sum_{k=0}^{[nt]-1}\tilde m^\beta_{\sigma^k\om}\circ \hat T_{\sigma^{nT}\om}^{(k-nT)}\right).
\end{split}
\]
Taking into account that $(\hat T_\omega^{(nT)})_*\hat \mu_{\omega}=\hat \mu_{\sigma^{nT}\omega}$ 
and that $\hat T_{\omega}^{(nT)}=(\hat T_{\sigma^{nT}\omega}^{(-nT)})^{-1}$, in order to complete the proof of the lemma it is enough to show that
\begin{equation}\label{F30}
\sup_{t\in [0, T]}\lvert F_{\omega, n}^3(t)\circ \hat T_{\om}^{(nT)}\rvert  \to 0, \,\hat\mu_\om \text{ a.s.}.   
\end{equation}
To prove \eqref{F30} let 
$$
\mathcal M_{\omega, n}(t)=\sum_{k=0}^{[nt]-1}\tilde m^\beta_{\sigma^k\om}\circ \hat T_{\sigma^{nT}\om}^{(k-nT)}.
$$
Then $\mathcal M_{\omega, n}(t)\circ \hat T_\om^{(nT)}$
is a martingale on the space $(\hat M_\omega, \hat{\mathcal B}_\omega, \hat\mu_\omega)$. 
Denote 
$$
S_{\omega,n}=\sup_{t\in[0,T]}|\mathcal M_{\omega,n}(t)\circ \hat T_\om^{(nT)}|,
$$
which for every $n$ and $\omega$ is a function on $\hat M_\omega$. Let $p>4$. 
Then by Doob's martingale inequality, 
there is a constant $C_p>0$ which depends only on $p$ such that
$$
\left\|S_{\omega,n}\right\|_{L^p(\hat \mu_{\omega})}\leq C_p\|\mathcal M_{\omega,n}(T)\circ \hat T_\om^{(nT)}\|_{L^p(\hat\mu_{\omega})}.
$$
Next, by Burkholder's inequality, 
for each fixed $\om$ we have 
\[
\begin{split}
\|\mathcal M_{\om,n}(T)\|_{L^p(\hat \mu_{\sigma^{nT}\omega})}&\leq\left\|\sum_{k=0}^{nT-1}(\tilde m^\beta_{\sigma^k\om}\circ \hat T_{\sigma^{nT}\om}^{(k-nT)})^2\right\|_{L^{p/2}(\hat \mu_{\sigma^{nT}\omega})}^{1/2}\\
&\leq\left(\sum_{k=0}^{nT-1}\|(\tilde m^\beta_{\sigma^k\om})^2\|_{L^{p/2}(\hat \mu_{\sigma^k \omega})}\right)^{1/2}.
\end{split}
\]
Now, let us write $p=4+\epsilon$ for some $\epsilon>0$. 
Since $\|m_\omega \|_{L^\infty(\mu_\om)}\in L^2(\Omega, \mathcal F, \mathbb P)$ and $\|(m_\om^\beta)^2\|_{L^{p/2}(\mu_\om)}\le \|m_\om^\beta \|_{L^\infty(\mu_\om)}^2$,
for $\mathbb P$-a.e. $\omega \in \Omega$ we have
$$
\sum_{k=0}^{n-1}\|(m^\beta_{\sigma^k\omega})^2\|_{L^{p/2}( \mu_{\sigma^k \omega})}=O(n) .
$$
Thus, for $\mathbb P$ a.e. $\omega$ we have
$$
\|\mathcal S_{\omega,n}\|_{L^p(\hat\mu_{\omega})}=O(\sqrt n)
$$
and so by the Markov inequality,
$$
\hat\mu_{\omega}(|\mathcal S_{\omega,n}|\geq n^{3/4})=
\hat\mu_{\omega}(|\mathcal S_{\omega, n}|^p\geq n^{3p/4})\leq C n^{-p/4}=Cn^{-1-\epsilon/4}.
$$
Thus, by the Borel-Cantelli lemma applied with respect to the probability measure $\hat \mu_\omega$ for $\mathbb P$ a.e. $\omega$ we have
$$
\mathcal S_{\om,n}=O(n^{3/4}), \quad \hat\mu_\om \,\text{ a.s.}
$$
Finally,
$$
\sup_{t\in[0,T]}\left|\tilde m_{\sigma^{[nt]}\om}^\gamma\circ \hat T_{\om}^{[nt]}\right|\leq 
\sup_{t\in[0,T]}\|\tilde m^\gamma_{\sigma^{[nt]}\omega}\|_{L^\infty(\hat\mu_{\sigma^{[nt]}\omega})}
$$
$$
=\sup_{t\in[0,T]}\| m^\gamma_{\sigma^{[nt]}\omega}\|_{L^\infty(\mu_{\sigma^{[nt]}\omega})}=o(n^{1/4}),
$$
where the last step follows since $\omega\to \| m^\gamma_\om\|_{L^\infty(\mu_\omega)}$ belongs to $L^4(\Omega,\mathcal F,\mathbb P)$. 
Combining the last two estimates and using second expression for $F_{\om,n}^3(t)$ we obtain \eqref{F30}.

To estimate $E_{\om,t}(t)$ by $\tilde{\mathbb M}_{\sigma^{nT} \omega, n}^{\gamma \beta, -}(T-t)$ in the case when $[n(T-t)]=Tn-[tn]-1$ we first note that we need to add to the previous expression for  $E_{\om,t}(t)-\tilde{\mathbb M}_{\sigma^{nT} \omega, n}^{\gamma \beta, -}(T-t)$  the following term 
\[
\begin{split}
R_{\omega, n}&=\frac 1 n\sum_{j=[nt]-nT+2}\tilde m_{\sigma^j(\sigma^{nT}\om)}^\gamma \circ \hat T_{\sigma^{nT}\omega}^{(j)}\tilde m^\beta_{\sigma^{[nt]} \om} \circ T_{\sigma^{nT}\om}^{([nt]-NT)}\\
&\phantom{=}+\frac 1 n \tilde m^\beta_{\sigma^{[nt]}\om}\circ \hat T_{\sigma^{nT}\om}^{([nt]-nT)}\cdot \tilde m^\gamma_{\sigma^{[nt]+1}\om}\circ \hat T_{\sigma^{nT}\om}^{([nt]-nT+1)}.
\end{split}
\]
Arguing like in the previous parts of the proof it will follow that 
\[
\sup_{t\in [0, T]}\lvert R_{\omega, n}(t)\circ \hat T_{\om}^{([nT])}\rvert  \to 0, \,\hat\mu_\om \text{ a.s.}
\]
Moreover,  note that when $[n(T-t)]=Tn-[tn]-1$ there is a similar correction term $\mathcal R_{\om,n}(t)$ in the second formula for $H_{\om,t}(t)$, which will also satisfy 
\[
\sup_{t\in [0, T]}\lvert \mathcal R_{\omega, n}(t)\circ \hat T_{\om}^{([nT])}\rvert  \to 0, \,\hat\mu_\om \text{ a.s.}
\]
\end{proof}

\begin{cor}\label{corr}
In the circumstances of both Lemma~\ref{l143} and Lemma \ref{l142}, 
for $\mathbb P$-a.e. $\omega \in \Omega$, 
\[
(\tilde M_{\omega, n}, \tilde{\mathbb M}_{\omega, n}) \to_w (g(W), (g(J)-h(W))^*) \quad \text{in $D([0, T], \R^e \times \R^{e\times e})$ when $n\to \infty$.}
\]
\end{cor}

\begin{proof}
The desired conclusion follows from Lemmas~\ref{l143} and~\ref{l142} by arguing as in the proof of~\cite[Corollary 4.10]{KM}.
\end{proof}

The following is the main result of this subsection.
\begin{lemma}\label{411}
Suppose that $\omega \mapsto \|m_\omega\|_{L^\infty(\mu_\omega)}\in L^{4}(\Omega, \mathcal F, \mathbb P)$. 
For $\mathbb P$-a.e. $\omega \in \Omega$, we have that $(M_{\omega, n}, \mathbb{M}_{\omega, n})\to_w (W, J)$ in $D([0, \infty), \R^e \times \R^{e\times e})$.
\end{lemma}

\begin{proof}
The desired conclusion follows from Corollary~\ref{corr} and~\cite[Lemma 4.11]{KM}.
\end{proof}

\subsection{Iterated weak invariance principle via martingale reduction}
We are now in a position to establish the iterated weak invariance principle for observables admitting an appropriate martingale decomposition. More precisely, we have the following result.

\begin{theorem}\label{iteratedprinciple}
Let $v\colon \Omega \times M\to \R^e$ be a measurable map such that~\eqref{spaceH} and~\eqref{centering} hold for $\mathbb P$-a.e. $\omega \in \Omega$. We suppose that $v$ admits a decomposition~\eqref{mdec}, 
for $\mathbb P$-a.e. $\omega \in \Omega$. Moreover, we require that~\eqref{lm} holds for  $\mathbb P$-a.e. $\omega \in \Omega$. Finally, we assume that there are $p>0$, $r\geq \frac{2q_0}{q_0-2}$ and $s\geq 4$ satisfying~\eqref{wer} and such that $K\in L^r(\Omega, \mathcal F, \mathbb P)$ and $\omega \mapsto \|v_\omega\|_{\mathcal H}\in L^p(\Omega, \mathcal F, \mathbb P)$, where $v_\omega:=v(\omega, \cdot)$. 

Then,
for $\mathbb P$-a.e. $\omega \in \Omega$, we have that $(W_{\omega, n}, \mathbb W_{\omega, n})\to_w (W, \mathbb W)$ in $D([0, \infty), \R^e \times \R^{e\times e})$, where:
\begin{enumerate}
\item[(i)] $W$ is an $e$-dimensional Brownian motion with covariance matrix $\Sigma=\Cov (W(1))=\lim_{n\to \infty} \Cov_{\mu} (W_n (1))$ given by
\begin{equation}\label{sig}
\Sigma^{\beta \gamma}=\int_{\Omega \times M}v^\beta v^\gamma\, d\mu+\sum_{n=1}^\infty \int_{\Omega \times M}(v^\beta v^\gamma \circ \tau^n+v^\gamma v^\beta \circ \tau^n)\, d\mu,
\end{equation}
where
\[
W_n(t)=\frac{1}{\sqrt n}\sum_{j=0}^{[nt]-1}v\circ \tau^j.
\]
\item [(ii)]
$\mathbb W^{\beta \gamma}(t)=\int_0^t W^\beta \, d W^\gamma +E^{\beta \gamma}t$, where $E=\lim_{n\to \infty}\mathbb E_{\mu}(\mathbb W_{n}(1))$ is given by 
\begin{equation}\label{ee}
E^{\beta \gamma}=\sum_{n=1}^\infty \int_{\Omega \times M} v^\beta v^\gamma \circ \tau^n \, d\mu,
\end{equation}
and
\[
\mathbb W^{\beta \gamma}_n(t)=\frac{1}{n}\sum_{0\le i<j\le [nt]-1}v^\beta \circ \tau^i \cdot v^\gamma \circ \tau^j.
\]
\end{enumerate}
\end{theorem}
\begin{proof}
The proof follows from Lemma~\ref{411} together with Corollary~\ref{co}, using the same reasoning as in the proof of~\cite[Theorem 4.3]{KM}.
\end{proof}

\section{Examples}

\subsection{Uniform decay of correlations}\label{UDC}
We discuss the case when $K$ in \eqref{NewExpConv} is  a constant variables and $A_n$ in \eqref{NewExpConv} has the form $A_n=\rho^n$ for some constant $\rho<1$. In this case ~\eqref{rho norm} holds for each $q_0\ge 1$.
For explicit examples we refer to~\cite[Section 2.3.1]{DFGTV2} (see also~\cite{DFGTV1, DH}), which include random Lasota-Yorke maps as well as random piecewise expanding maps in higher dimension. The associated space $\mathcal H$ is the space of functions of bounded variation (see~\cite[Section 2.2]{DFGTV2}). We emphasize  that since the arguments in the present paper do not rely on the multiplicative ergodic theory, it is not necessary to assume that the map $\omega \mapsto T_\omega$ is countably valued as in~\cite{DFGTV2}. We stress that for these examples, the crucial assumption is that $\essinf_\om \gamma_\om>1$, where $\gamma_\om$ is
the 
minimal amount of local expansion of $T_\omega$ (see~\cite[Eq. (20)]{DFGTV2}).

\subsubsection{Application to homogenization}

For given
$a\colon \R^d \to \R^d$, $b\colon \R^d \to \R^{d\times e}$, $v\colon \Omega \times M  \to \mathbb R^e$, $\omega \in \Omega$ and $\epsilon >0$, we consider the following slow-fast system:
\begin{equation}\label{sfs}
\begin{split}
x_{n+1}&=x_n+\epsilon^2 a(x_n)+\epsilon b(x_n)v_{\sigma^n \omega}(y_n) \\
y_{n+1}&=T_{\sigma^n \omega}(y_n).
\end{split}
\end{equation}
We note that the above system is possed on $\R^d \times M$. We refer to the first equation in~\eqref{sfs} as to the slow component (since we will be interested in the case when $\epsilon$ is close to $0$), while the second equation in~\eqref{sfs} will be called the fast component.

 We observe that the solution of~\eqref{sfs} with an initial condition $x(0)=\xi$ is given by 
\[
x_n^\epsilon=\xi+ \epsilon^2 \sum_{j=0}^{n-1}a(x_j^\epsilon)+\epsilon \sum_{j=0}^{n-1}b(x_j^\epsilon)v_{\sigma^j \omega}(y_j), \quad y_n=T_\omega^{(n)}y_0.
\]
Set $\hat{x}^\epsilon (t)=x_{[t/\epsilon^2]}^\epsilon$.
\begin{remark}
We observe that $\hat{x}^\epsilon (t)$ depends also on $\omega$ but in order to keep the notation as simple as possible, we don't make this dependence explicit. 
\end{remark}
We will assume that $v$ satisfies~\eqref{spaceH}, \eqref{centering} and  \begin{equation}\label{tuu}\esssup_{\om \in \Omega}\|v_\om\|_{\mathcal H}<+\infty.\end{equation} Since~\eqref{rho norm} holds with $q_0=\infty$ and $K\in L^\infty$, we have that \begin{equation}
\label{mchi}\esssup_{\omega \in \Omega}\|m_\omega\|_{L^\infty(\mu_\om)}<+\infty \quad  \text{and} \quad \esssup_{\omega \in \Omega}\|\chi_\omega\|_{L^\infty(\mu_\om)}<+\infty.\end{equation}

We have the following homogenization result. 
\begin{theorem}\label{homogenization}
Suppose that $a\colon \R^d \to \R^d$ is $C^2$ and that $b\colon \R^d \to \R^{d\times e}$ is $C^3$. For $\mathbb P$-a.e. $\omega \in \Omega$, we have that $\hat{x}^\epsilon \to_w Z$ as $\epsilon \to 0$, where $Z$ is a solution of the stochastic differential equation
\[
dZ=\tilde a(Z)dt+b(Z)dW, 
\]
where $W$ is a $e$-dimensional Brownian motion with covariance matrix $\Sigma$ and
\[
\tilde a (x)=a(x)+\sum_{\alpha=1}^d\sum_{\beta, \gamma=1}^eE^{\beta \gamma}\frac{\partial b^\beta}{\partial x_\alpha}b^{\alpha \gamma}(x),
\]
and $\Sigma$ and $E$ are given by~\eqref{sig} and~\eqref{ee} respectively.
\end{theorem}

To prove the above theorem we first need the following two types of moment estimates. In the sequel, $C$ will denote a generic positive constant independent on $\omega$ that can change its value from one occurrence to the next.
\begin{proposition}\label{Mom prop 1}
Let $p> 2$. For $\mathbb P$-a.e. $\om \in \Omega$,
\[
\left\|\max_{k\leq n}|S_k^\om v|\right\|_{L^p(\mu_\om)} \le Cn^{1/2},
\]
where $S_n^\om v=\sum_{j=0}^{n-1}v_{\sigma^j\om}\circ T_\om^{(j)}.$
\end{proposition}

\begin{proof}
It follows from~\eqref{mdec} and~\eqref{mchi} that it is sufficient to prove that
\begin{equation}\label{zu8}
\left\|\max_{k\leq n}|S_k^\om m|\right\|_{L^p(\mu_\om)}\le Cn^{1/2},
\end{equation}
for $\mathbb P$-a.e. $\om \in \Omega$.
Next, by Burkholder's inequality we have 
$$
\left\|\max_{k\leq n}|S_n^\om m|\right\|_{L^{p}(\mu_\om)}\leq\|S_{n}^\om (m^2)\|_{L^{p/2}(\mu_\om)}^{1/2}. 
$$
Moreover, 
\[
\|S_{n}^\om (m^2)\|_{L^{p/2}(\mu_\om)}\le \sum_{i=0}^{n-1} \|m_{\sigma^i \om}\|_{L^p(\mu_{\sigma^i \om})}^2 \le 
\sum_{i=0}^{n-1} \|m_{\sigma^i \om}\|_{L^\infty(\mu_{\sigma^i \om})}^2
\le Cn,
\]
where in the last step we used~\eqref{mchi}. The last two estimates readily imply~\eqref{zu8}.
\end{proof}

\begin{proposition}\label{mel2}
 For $\mathbb P$-a.e. $\omega$ 
and every $q>2$ we have that 
\[
\bigg{\lVert} \max_{k\le n}\lvert \mathbb S_{\omega, k}^{\beta \gamma}\rvert \bigg{\rVert}_{L^{q/2}(\mu_\omega)}\le C n,
\]
where
\[
\mathbb S_{\omega, n}^{\beta \gamma}=\sum_{0\le i <j<n} v_{\sigma^i \omega}^\beta \circ T_\omega^{(i)} \cdot v_{\sigma^j \omega}^\gamma \circ  T_\omega^{(j)},
\]
for $\omega \in \Omega$, $n\in \N$ and $\beta, \gamma \in \{1, \ldots, e\}$.
\end{proposition}

\begin{proof}
Using~\eqref{mdec} we see that 
\[
\begin{split}
\mathbb S_{\omega, n}^{\beta \gamma}&=\sum_{0\le i <j<n} m_{\sigma^i \omega}^\beta \circ T_\omega^{(i)} \cdot v_{\sigma^j \omega}^\gamma \circ  T_\omega^{(j)}+\sum_{1\le j<n}(\chi_{\sigma^j \omega}^\beta \circ T_\omega^{(j)}-\chi_\omega^\beta)v_{\sigma^j \omega}^\gamma \circ  T_\omega^{(j)} \\
&=I_{\omega, n}+J_{\omega, n},
\end{split}
\]
where
\[
I_{\omega, n}:=\sum_{0\le i <j<n} m_{\sigma^i \omega}^\beta \circ T_\omega^{(i)} \cdot m_{\sigma^j \omega}^\gamma \circ  T_\omega^{(j)}
\]
and 
\[
J_{\omega, n}:=\sum_{0\le i <n-1}m_{\sigma^i \omega}^\beta \circ T_\omega^{(i)} (\chi_{\sigma^n \omega}^\gamma \circ T_\omega^{(n)}-\chi_{\sigma^{i+1}\omega}^\gamma \circ T_\omega^{(i+1)})+\sum_{1\le j<n}(\chi_{\sigma^j \omega}^\beta \circ T_\omega^{(j)}-\chi_\omega^\beta)v_{\sigma^j \omega}^\gamma \circ  T_\omega^{(j)}.
\]
Observe that 
\begin{equation}\label{523}
\max_{k\le n}\lvert J_{\omega, k}|\leq
I_{1}(\om,n)+I_{2}(\om,n)+I_3(\om,n)+I_4(\om,n)
\end{equation}
where 
$$
I_1(\om,n)=
\max_{k\leq n}\left(|\lvert \chi_{\sigma^k \omega}^\gamma \rvert|_{L^\infty(\mu_{\sigma^k\om})}\left|\sum_{i=0}^{k-1} ( m_{\sigma^i \omega}^\beta  \circ T_\omega^{(i)})\right|\right),
$$
$$
I_2(\om,n)=
\sum_{0\le i <n-1} |\lvert m_{\sigma^i \omega}^\beta \rvert|_{L^\infty(\mu_{\sigma^i\om})}  \cdot   |\lvert \chi_{\sigma^{i+1}\omega}^\gamma \rvert|_{L^\infty(\mu_{\sigma^{i+1}\om})},
$$
$$
I_{3}(\om,n)=\sum_{1\le j<n}|\lvert \chi_{\sigma^j \omega}^\beta \rvert|_{L^\infty(\mu_{\sigma^j\om})}  \cdot |\lvert v_{\sigma^j \omega}^\gamma \rvert|_{L^\infty(\mu_{\sigma^j\om})},
$$
and
$$
I_4(\om,n)= |\lvert \chi_\omega^\beta \rvert|_{L^\infty(\mu_{\om})} \sum_{1\le j<n}|\lvert v_{\sigma^j \omega}^\gamma \rvert|_{L^\infty(\mu_{\sigma^j\om})}.
$$
Taking into account \eqref{mchi} it is clear that 
\begin{equation}\label{e1}
\bigg{\lVert} \max_{k\le n}\lvert J_{\omega, k}\rvert \bigg{\rVert}_{L^{q/2}(\mu_\omega)} \le C n, \quad \text{for $\mathbb P$-a.e. $\omega \in \Omega$.}
\end{equation}
In order to estimate the contribution of $I_{\omega, n}$ we notice that 
$$
I_{\omega,n}=\sum_{k=1}^{n-1}Y_{\omega,k,n},
$$
where 
$$
Y_{\omega,k, n}=m_{\sigma^{n-k}\omega}^\gamma \circ T_\om^{(n-k)}\sum_{j=0}^{n-k-1}m_{\sigma^j\omega}^\beta \circ T_\omega^{(j)}.
$$
Notice that for $\mathbb P$-a.e. $\omega \in \Omega$ and every fixed $n$, the finite 
sequence $\{Y_{\omega,k,n}: 1\leq k<n\}$ is a reversed martingale difference. 
 By applying the inequalities of Doob and
Burkholder for reversed martingales we see that, writing $Y_k=Y_{\omega,k,n}$, 
\[
\begin{split}
\lVert \max_{k\le n} \lvert I_{\omega, k}\rvert \rVert_{L^{q/2}( \mu_\omega)}^2 \le C_q \lVert I_{\omega, n}\rVert_{L^{q/2}(\mu_\omega)}^2 &\le C_q \bigg{\lVert} \bigg (\sum_{k=1}^{n-1}  Y_k^2 \bigg )^{1/2} \bigg \rVert_{L^{q/2}(\mu_{\omega})}^2 \\
&=C_q\bigg{\lVert} \sum_{k=1}^{n-1}  Y_k^2 \bigg \rVert_{L^{q/4}( \mu_{\omega})},
\end{split}
\]
where $C_q>0$ depends only on $q$. Therefore,
\[
\lVert \max_{k\le n} \lvert I_{\omega, k}\rvert \rVert_{L^{q/2}( \mu_\omega)}^2 \le C_q\sum_{k=1}^{n-1} \lVert Y_k\rVert_{L^{q/2}(\mu_{\omega})}^2.
\]
Using the H\"{o}lder inequality, we obtain that 
\[
\begin{split}
\lVert Y_k\rVert_{L^{q/2}(\mu_{\omega})} &\le \bigg{\lVert} \sum_{j=0}^{n-k-1} m_{\sigma^{j}\omega}^\beta \circ  T_{\omega}^{(j)}\bigg{\rVert}_{L^q ( \mu_{\omega})} \cdot \lVert  m^\gamma_{\sigma^{n-k}\omega}\circ T_{\omega}^{(n-k)}\rVert_{L^q (\mu_{\omega})}\\
&\leq C\bigg{\lVert} \sum_{j=0}^{n-k-1} m_{\sigma^{j}\omega}^\beta \circ  T_{\omega}^{(j)}\bigg{\rVert}_{L^q ( \mu_{\omega})},
\end{split}
\]
where in the last inequality we have used  \eqref{mchi}.
Now, by Proposition~\ref{Mom prop 1} (see~\eqref{zu8}) we have that 
$$
 \bigg{\lVert} \sum_{j=0}^{n-k-1} m_{\sigma^{j}\omega}^\beta \circ  T_{\omega}^{(j)}\bigg{\rVert}_{L^q ( \mu_{\omega})}\leq C\sqrt{n-k}.
$$
Combining the above estimates we conclude that
\[
\lVert \max_{k\le n} \lvert I_{\omega, k}\rvert \rVert_{L^{q/2}(\mu_\omega)}^2 \leq C\sum_{k=1}^{n-1}(n-k)\leq Cn^2.
\]
We conclude that 
$$
\lVert \max_{k\le n} \lvert I_{\omega, k}\rvert \rVert_{L^{q/2}(\mu_\omega)}\leq Cn\quad  \text{ $\mathbb P$-a.s.,}
$$
and the proof of the proposition is completed.

\end{proof}

Next, for $\omega \in \Omega$ and $n\in \N$ we define
\[
W_{\omega, n}(s, t)=W_{\omega, n}(t)-W_{\omega, n}(s) \quad \text{and} \quad \mathbb W_{\omega, n}(s,t)=\mathbb W_{\omega, n}(t)-\mathbb W_{\omega, n}(s),
\]
for $s, t\ge 0$.  The following result is a consequence of the previous moment estimates.
\begin{cor}\label{tzzz}
Let $v\colon \Omega \times M\to \R^e$ be an observable satisfying~\eqref{spaceH}, \eqref{centering}, and~\eqref{tuu}. Then, for every $p> 2$ there exists $C>0$ such that for $\mathbb P$-a.e. $\omega \in \Omega$ we have that
\begin{equation}\label{1234}
\lVert W_{\omega, n}(j/n, k/n)\rVert_{L^p (\mu_\omega)}\le C(\lvert k-j\rvert /n)^{1/2} \quad \text{and} \quad \lVert \mathbb W_{\omega, n}(j/n, k/n)\rVert_{L^{p/2} (\mu_\omega)} \le C\lvert k-j\rvert /n, 
\end{equation}
for $j, k, n\in \N$.
\end{cor}
The role of this corollary is to ensure tightness in an appropriate space of H\"older functions (after passing to appropriate piecewise continuous versions of the random functions), which is needed in order to apply rough path theory.
\begin{proof}
Observe that for $t>s>0$ we have that 
\[
\begin{split}
W_{\omega, n}(s, t) &=\frac{1}{\sqrt n}\sum_{j=0}^{[nt]-1}v_{\sigma^j \omega}\circ T_\omega^{(j)}-\frac{1}{\sqrt n}\sum_{j=0}^{[ns]-1}v_{\sigma^j \omega}\circ T_\omega^{(j)} \\
&=\frac{1}{\sqrt n}\sum_{j=[ns]}^{[nt]-1}v_{\sigma^j \omega}\circ T_\omega^{(j)} \\
&=\frac{1}{\sqrt n} \bigg{(}\sum_{j=0}^{[nt]-[ns]-1}v_{\sigma^j (\sigma^{[ns]} \omega)}\circ T_{\sigma^{[ns]}\omega}^{(j)}\bigg{)}\circ T_\omega^{([ns])}.
\end{split}
\]
Let us now assume without any loss of generality that $j<k$. Hence,
\[
W_{\omega, n}(j/n, k/n)=\frac{1}{\sqrt n}\bigg{(}\sum_{i=0}^{k-j-1}v_{\sigma^i (\sigma^j \omega)}\circ T_{\sigma^j \omega}^{(i)}\bigg{)}\circ T_\omega^{(j)},
\]
and consequently the first inequality in~\eqref{1234} follows readily from Proposition~\ref{Mom prop 1}.

Similarly, we obtain  that
\[
\mathbb W_{\omega, n}(j/n, k/n)=\frac 1 n \mathbb S_{\sigma^j \omega, k-j}\circ T_\omega^{(j)}, 
\]
and thus the second inequality~\eqref{1234} follows from Proposition~\ref{mel2}. The proof of the corollary is completed.
\end{proof}

\begin{remark}\label{failure}
Under conditions similar to the one in Theorem \ref{T1}, 
 in the non-uniform case we can show that 
$\|\max_{k\leq n}|S_k^\om v|\|_{L^p(\mu_\om)} \le C_\om n^{1/2}$ and  $\|\max_{k\leq n}|\mathbb S_{\om,k}^{\beta\gamma}|\|_{L^{p/2}(\mu_\om)}\leq C_\om n$ for some random variable $C_\om$, whenever $p$ is sufficiently large. However, this is not enough for the conclusion of the previous corollary to hold. 
\end{remark}
The conclusion of Theorem~\ref{homogenization} now follows from Theorem~\ref{iteratedprinciple} and Corollary~\ref{tzzz} by applying~\cite[Theorem 4.10]{CFKMZ0}. In fact, Theorem~\ref{iteratedprinciple} verifies~\cite[Assumption 4.6.]{CFKMZ0}, while Corollary~\ref{tzzz} shows that~\cite[Assumption 4.7.]{CFKMZ0} is valid (with any $q>1$). 

\subsection{Nonuniform decay of correlations I: Non-uniformly expanding maps}

We assume that there are random variables $\xi_\om \in(0,1]$ and
$\gamma_\om>1$  such that, $\mathbb P$-a.s.
for every $x,x'\in  M$ with $d(x,x')\leq \xi_{\sigma\om}$ we can write 
\begin{equation}\label{Pair1.0}
T_\om^{-1}\{x\}=\{y_i=y_{i,\om}(x): i<k\}\,\,\text{ and }\,\,T_\om^{-1}\{x'\}=\{y_i'=y_{i,\om}(x'): i<k\}
\end{equation}
and we have
\begin{equation}\label{Pair2.0}
d(y_i,y_i')\leq (\gamma_\omega)^{-1}d(x,x')
\end{equation}
for all  $1\leq i<k=k(\omega,x)$ (where either $k\in\mathbb N$ or $k=\infty$).  We refer to \cite[Section 3]{YH 2023} for several concrete examples of maps. Here $d(\cdot,\cdot)$ is the metric on $M$.

For base maps $\sigma$ satisfying some mixing related conditions and under several integrability 
and approximation  conditions on random variables like $\deg(T_\omega)$ (when finite), the random H\"older constant of $T_\omega$ with respect to a given exponent $\alpha\in(0,1]$  etc.,
and some expansion on average assumptions,
in \cite[Theorem 2.13]{YH 2023} we proved the following. For every random H\"older continuous function $\phi_\omega:X\to \mathbb R$ with exponent $\alpha$ and $\log^+$ integrable random H\"older norm there is a unique family of equivariant measures $\mu_\omega$ which satisfies an appropriate Gibbs property, and it is also a random equilibrium state. Moreover, when $\phi_\om=-\ln \text{Jac}(T_\om)$, $\mu_\om$ is the unique family of random equivariant measures which are absolutely continuous with respect to the volume measure.

Moreover when $\omega\mapsto\|\phi_\omega\|_\alpha$ satisfies appropriate moment, mixing and approximation conditions,
in \cite[Theorem 2.13]{YH 2023} (i) we showed that
$$
\sup_{\|g\|_\alpha\leq 1} \left \|L_\omega^ng-\int_M g\, d \mu_\omega \right \|_\infty\leq R(\omega)n^{-\beta},
$$
for some $\beta>1$ and a random variable $R\in L^t(\Omega,\mathcal F,\mathbb P)$, where $t$ is a parameter that depends on the assumptions of \cite[Theorem 2.13]{YH 2023}, and under the right assumptions in can be taken to be arbitrarily large. Here $\|\cdot\|_\alpha$ is the usual H\"older norm corresponding to the exponent $\alpha$.
Thus \eqref{NewExpConv} and \eqref{rho norm} hold with $q_0=t$, a constant $A_n=n^{-\beta}$ and the norm $\|\cdot\|_{\mathcal H}=\|\cdot\|_\alpha$.

\subsection{Nonuniform decay of correlations II: random maps with dominating expansion}
Here we return to the setup of \cite[Section 2.2]{YH}. 
We suppose that there exist random variables $l_\omega\geq 1$, $\eta_\omega>1$, $q_\omega\in\N$ and $d_\omega\in\N$ so that 
$q_\omega<d_\omega$ and for every $x\in M$ we can write 
\begin{equation}\label{Pair1}
T_\omega^{-1}\{x\}=\{y_{1,\omega}(x),\ldots, y_{d_\omega,\omega}(x)\}
\end{equation}
where for every $x,x'\in M$ and for  $i=1,2,\ldots ,q_\omega$ we have
\begin{equation}\label{Pair2}
d(y_{i,\omega}(x),y_{i,\omega}(x'))\leq l_\omega d(x,x')
\end{equation}
while for  $i=q_{\omega}+1,\ldots ,d_\omega$,
\begin{equation}\label{Pair3}
d(x_i,x'_i)\leq \eta_\omega^{-1}d(x,x').
\end{equation}

Let us fix some $\alpha\in(0,1]$ and 
assume that
\begin{equation}\label{a om}
a_\om:=\frac{q_\omega l_\omega^\alpha+(d_\omega-q_\omega)\eta_\omega^{-\alpha}}{d_\omega}<1,
\end{equation}
which is a quantitative estimate on the amount of allowed contraction, given the amount of expansion $T_\omega$ has.

Now, let $\phi_\omega$ be a random H\"older continuous function with respect to the exponent $\alpha$. Let
$$
\varepsilon_\om=\text{osc}(\phi_\omega)=\sup\phi_\omega-\inf\phi_\omega.
$$
Let also
\begin{equation}\label{H def}
H_\omega=\max\{v_\alpha(\phi_\omega\circ y_{i,\omega}):\,1\leq i\leq d_\omega\}
\end{equation}
where $v_\alpha(g)$ denotes the $\alpha$-H\"older constant of a function $g$.
Our additional requirements  related to the function $\phi_\omega$ is that
\begin{equation}\label{Bound ve}
s_\om:=e^{\varepsilon_\omega}a_\omega<1\,\,\text{ and }\,\,e^{\varepsilon_\omega}H_\om\leq\frac{s_{\sigma\omega}^{-1}-1}{1+s_\omega^{-1}}.
\end{equation}
In~\cite[Theorem 47]{YH} we showed that the random Gibbs measures corresponding to the random potential has the property that
\begin{equation}\label{Eg Exp}
\left \|L_\omega^{(n)}(\varphi)-\int_M\varphi \, d\mu_\omega \right \|_{\alpha}\leq 
B(\sigma^n\omega)\rho_{\omega,n}   \| \varphi \|_{\alpha},
\end{equation}
for $\varphi \in \mathcal H$, $n\in \N$ and $\mathbb P$-a.e. $\omega \in \Omega$.
 Here 
 $$
\rho_{\om,n}=\prod_{j=0}^{n-1}\rho(\sigma^j\omega)
 $$
 and $\rho(\om)<1$ has an explicit form given in \cite[Section 2.3.2]{YH}. Moreover, $B(\om)=12(1+2/s_\om)^4$.

\begin{lemma}\label{B lemma}
Suppose that  $B\in L^q(\Omega,\mathcal F,\mathbb P)$ for some $q>0$.  Then, for every sequence of positive numbers $(a_n)$ such that $\sum_{n\geq 1}a_n^q<\infty$ there is a random variable $R\in L^q(\Omega,\mathcal F,\mathbb P)$ such that
\begin{equation}\label{BR}
B(\sigma^n\omega)\leq R(\omega)a_n^{-1},
\end{equation}
for $\mathbb P$-a.e. $\omega \in \Omega$ and all $n\in \N$.
Moreover, 
\[
\|R\|_{L^q(\Omega, \mathcal F, \mathbb P)}^q\leq \|B\|_{L^q(\Omega, \mathcal F, \mathbb P)}^q\sum_{n\geq 1}a_n^q.
\]
\end{lemma}

\begin{proof}
Let 
$$
R(\omega)=\sup_n (a_nB(\sigma^n\omega)), \quad \omega \in \Omega.
$$ 
Clearly,
$$
(R(\omega))^q\leq \sum_{n\geq 1}a_n^q(B(\sigma^n\omega))^q,
$$
and therefore
$$
\int_\Omega R^q\, d\mathbb P\leq \left (\int_\Omega B^q\, d\mathbb P\right )\sum_{n\geq 1}a_n^q<\infty,
$$
which implies the desired conclusions.
\end{proof}

We now show how one can verify~\eqref{rho norm} in the case when $\rho$ is not a constant (actually, we will need to slightly modify $\rho$ in the case $B$ is not bounded). Later on we will show that it is sufficient to show that $\|\rho_{\omega,n}\|_{L^q}$ decays sufficiently fast as $n\to\infty$ for appropriate $q$'s. 
We refer to \cite[Lemma 4.4]{YH 2023}, 
\cite[Lemma 4.6]{YH 2023}, \cite[Corollary 4.7]{YH 2023}  and \cite[Lemma 4.8]{YH 2023} for sufficient conditions for sufficiently fast decay of
$\|\rho_{\omega,n}\|_{L^q}$ as $n\to\infty$. In order to demonstrate the idea, let us include here a proof that
$\|\rho_{\omega,n}\|_{L^q}$ decays exponentially fast under appropriate conditions.

Let $X=(X_j)_{j\in \Z}$ be a stationary sequence of random variables (taking values on some measurable space) which generates the system
$(\Omega, \mathcal F, \mathbb P, \sigma)$, so that $\sigma$ is the left shift on the paths of $X_j$, i.e. $\sigma ((X_j)_j)=(X_{j+1})$. For $k\in \N$, let $\psi_U(k)$ be the smallest number with the property that 
\[
\mathbb P(A\cap B)\le \mathbb P(A)  \mathbb P(B) (1+\psi_U(k)),
\]
for all $A\in \sigma \{X_j: \ j\le n\}$ with $n\in \Z$ and $B\in \sigma\{X_j: \ j\ge n+k\}$. Here, $\sigma \{X_j: \ j\in \mathcal I\}$ denotes the $\sigma$-algebra generated by $X_j$, $j\in \mathcal I$ with $\mathcal I\subset \Z$. We note that $\psi_U$ is the so-called upper $\psi$-mixing coefficient. When the random variables $\{X_j: j\in\mathbb Z\}$ are i.i.d then $\psi_U(k)=0$ for all $k$, and so $\psi_U$ measure the dependence from above. 
We refer to \cite{YH} for many examples where $\psi_U(k)\to 0$ as $k\to\infty$. Actually in these examples the two-sided version $\psi(k)$ of $\psi_U(k)$ decays to $0$. These examples include many classes of Markov chains, and situations where $(X_j)_{j\in\mathbb Z}$ is distributed like a Gibbs measure on a topologically mixing subshift of finite type, as well as additional dynamical examples.

\begin{lemma}\label{rho norm-1}
Suppose  that 
\begin{equation}\label{UpperMixing}
\limsup_{k\to\infty}\psi_U(k)<\frac{1}{\mathbb E_{\mathbb P}[\rho]}-1,
\end{equation}
where $\mathbb E_{\mathbb P}[\rho]:=\int_\Omega \rho \, d\mathbb P<1$. In addition, assume that there is a sequence of positive numbers  $(\beta_r)_{r\in \N}$ with $\beta_r\to 0$, and a sequence of random variables $\rho_r \colon \Omega \to (0, \infty)$, $\rho_r$ measurable with respect to $\sigma \{X_j: \ |j| \le r\}$ such that
\begin{equation}\label{betar}
\|\rho-\rho_r\|_{L^\infty(\mathbb P)}\le \beta_r, \quad r\in \N.
\end{equation}
Then, for every $q\geq 1$,  there exists $\delta=\delta_q\in(0,1)$ such that $\|\rho_{\omega,n}\|_{L^q(\Omega,\mathcal F, \mathbb P)}\leq \delta^n$ for all $n\in \N$.
In particular, conditions \eqref{NewExpConv} and \eqref{rho norm} hold with $\|\cdot\|_{\mathcal H}=\|\cdot\|_\alpha$ for every $q_0\geq 1$.

\end{lemma}
\begin{proof}
Let us take $s\in\mathbb N$ of the form $s=3r$, where $r\in \N$ will be fixed later on. Then, since $\rho(\cdot)\in (0, 1)$ and $q\ge 1$, using \cite[Lemma 60]{YH} together with the fact that $\sigma$ preserves $\mathbb P$, we see that for $j\in \N$, 
\[
\|\rho_{\om, j}\|_{L^q(\Omega, \mathcal F, \mathbb P)}^q\le \mathbb E_{\mathbb P}[\rho_{\om, j}]\le
\mathbb E_{\mathbb P} \left[\prod_{k=0}^{[j/s]-1}\rho(\sigma^{ks}\om) \right]\leq (1+\Psi_U(r))^{[j/s]-1} \left (\mathbb E_{\mathbb P}[\rho] \right )^{[j/s]}.
\]
On the other hand, by~\eqref{betar} we have that 
\[
\mathbb E_{\mathbb P}[\rho]\le \mathbb E_{\mathbb P}[\rho_r]+\beta_r,
\]
and consequently,
\[
\|\rho_{\om, j+1}\|_{L^q(\Omega, \mathcal F, \mathbb P)}^q\le  (1+\Psi_U(r))^{[j/s]}a_r^{[j/s]},
\]
where $a_r:=\mathbb E_{\mathbb P}[\rho_r]+\beta_r$. Now, it follows easily from~\eqref{UpperMixing} that by taking $r$ sufficiently large, 
we can ensure that $\eta:=a_r(1+\Psi_U(r))<1$. Hence, 
$$
\|\rho_{\om, j+1}\|_{L^q(\Omega, \mathcal F, \mathbb P)}\leq \eta^{j/q},
$$
which implies that~\eqref{rho norm} holds with  $\delta:=\eta^{1/q}\in (0, 1)$. The proof of the lemma is completed.
\end{proof}

The following result is a consequence of the previous two lemmas. 
\begin{cor}
Under the assumptions of Lemma~\ref{rho norm-1} hold we have the following. 
Suppose that  $B\in L^q(\Omega, \mathcal F, \mathbb P)$ for  some $q>0$. Then,  for every $q_0\geq 1$ there exist a random variables $K_{q_0}\in L^q(\Omega, \mathcal F, \mathbb P)$ and  $\tilde \rho \colon \Omega \to (0, \infty)$ such that 
\begin{equation}\label{summab}
\sum_{n\geq 1}\|\tilde\rho_{\omega, n}\|_{L^{q_0}(\Omega, \mathcal F, \mathbb P)}<\infty
\end{equation}
and
\begin{equation}\label{2032}\left \|L_\omega^{(n)}(\varphi)-\int_M \varphi \, d\mu_\omega \right \|_{\infty}\leq K_{q_0}(\omega)  \tilde\rho_{\omega,n} \| \varphi \|_{\alpha},
\end{equation}
for all $\alpha$-H\"older continuous functions $\varphi:M\to\mathbb R$,  $n\in \N$ and $\mathbb P$-a.e. $\omega \in \Omega$.
\end{cor}
\begin{proof}
 Let us fix some $q_0\geq 1$ and let $\delta=\delta_{q_0}\in (0, 1)$ be given by Lemma~\ref{rho norm-1}. Write $\delta=e^{-a}$ with $a>0$ and choose $0<\epsilon<a$. By Lemma~\ref{B lemma} applied with  with $a_n=e^{-\epsilon}$ there is a random variable $R_{\epsilon}\in L^q(\Omega, \mathcal F, \mathbb P)$ such that~\eqref{BR} holds.
 Using~\eqref{Eg Exp} we have that 
 $$
\left \|L_\omega^{(n)}(\varphi)-\int_M \varphi \, d\mu_\omega \right \|_{\infty}\leq R_{\epsilon}(\omega)\prod_{j=0}^{n-1}(e^\epsilon \rho(\sigma^j\omega))\|\varphi\|_{\alpha},
 $$
 for $\varphi \in \mathcal H$, $n\in \N$ and $\mathbb P$-a.e. $\omega \in \Omega$.
 Therefore, by taking $K_{q_0}:=R_{\epsilon}$ and $\tilde \rho:=e^\epsilon \rho$, we have that~\eqref{2032} holds.  In addition, observe that for $n\in \N$,
 $$
\|\tilde\rho_{\omega,n}\|_{L^{q_0}(\Omega, \mathcal F, \mathbb P)}=e^{\epsilon n}\|\rho_{\omega,n}\|_{L^{q_0}(\Omega, \mathcal F, \mathbb P)}\leq e^{-(a-\epsilon)n},
 $$
and thus~\eqref{summab} holds.
\end{proof}
\begin{remark}
As noted in the discussion after Lemma \ref{B lemma} in other circumstances 
    $\mathbb E[\rho_{\om,n}]\leq K(\om)b_n$ for some $K\in L^q$ (where $q$  is as in Lemma \ref{B lemma}) and a sequence $(b_n)$, which decays either polynomially fast or (possibly stretched) exponentially fast. In this case by using Lemma \ref{B lemma} with an appropriate sequence we will get that the left hand side of \eqref{2032} is bounded by 
    $$
K(\om)R(\om)a_n^{-1}b_n
    $$
    where $R(\om)$ comes from Lemma \ref{B lemma}. Note that $K\cdot R\in L^{q/2}$. Thus conditions \eqref{NewExpConv} and \eqref{rho norm} will hold with $q_0=q/2$ if $\sum_{n}a_n^{-1}b_n<\infty$.
\end{remark}

\section*{Acknowledgements}
D. D. was supported in part by Croatian Science Foundation under
the project IP-2019-04-1239 and by the University of Rijeka under the projects
uniri-prirod-18-9 and uniri-prprirod-19-16. D. D. would like to thank I. Melbourne,  M. Pollicott and  S. Vaienti  for useful discussions related to the content of this paper that happened in November 2019 at Luminy. In particular, he is grateful to I. Melbourne for explaining his work~\cite{KM}.
We would like to thank the referee for his/hers detailed reading of our paper.

\section{Appendix}
\begin{proposition}\label{ergodicity}
Suppose that for each continuous function $\varphi \colon M\to \R$ and $\epsilon >0$, there exists $\psi \in \mathcal H$ such that 
$\sup_{x\in M}|\varphi(x)-\psi(x)| \le \epsilon$. Then, $\mu$ given by~\eqref{egmu} is ergodic.
\end{proposition}
\begin{proof}
We follow closely the proof of~\cite[Proposition 4.7]{MSU}.
Take a measurable $\mathcal C\subset \Omega \times M$ such that $\tau^{-1}(\mathcal C)=\mathcal C$. We need to show that $\mu (\mathcal C)\in \{0, 1\}$. For $\omega \in \Omega$, let
\[
\mathcal C_\omega:=\{x\in M: (\omega, x)\in \mathcal C\} \in \mathcal B.
\]
Observe that 
\[
x\in T_\omega^{-1}(\mathcal C_{\sigma \omega}) \iff T_\omega (x)\in \mathcal C_{\sigma \omega} \iff (\sigma \omega, T_\omega(x))\in \mathcal C \iff \tau(\omega, x)\in \mathcal C \iff (\omega, x)\in \mathcal C,
\]
which implies that 
\begin{equation}\label{eee}
    T_\omega^{-1}(\mathcal C_{\sigma \omega})=\mathcal C_\omega, \quad \omega \in \Omega.
\end{equation}
Set
\[
\Omega_0:=\{ \omega \in \Omega: \ \mu_\omega (\mathcal C_\omega)>0\}\in \mathcal F.
\]
By~\eqref{eee} we have that $\sigma(\Omega_0)=\Omega_0$. Since $\sigma$ is ergodic, we conclude that $\mathbb P(\Omega_0)\in \{0, 1\}$.  If $\mathbb P(\Omega_0)=0$, then clearly $\mu(\mathcal C)=0$. 

From now on we  suppose that $\mathbb P(\Omega_0)=1$. Without any loss of generality, we may suppose that the conclusion of Lemma~\ref{dec} holds for each $\omega \in \Omega_0$. Furthermore, in the view of~\eqref{rho norm} (which implies that $A_j\to 0$ in $L^q(\Omega, \mathcal F, \mathbb P))$,  we can assume that there is a subsequence $(n_j)_j$ of $\N$ such that $A_{n_j}(\omega)\to 0$ for $\omega \in \Omega_0$. We now claim that 
\begin{equation}\label{CLAIM}
\int_{\mathcal C_\omega}\varphi d\mu_\om=0, \quad \text{for $\omega \in \Omega_0$ and $\varphi \in \mathcal H$ such that $\int_M\varphi \, d\mu_\omega=0$.}
\end{equation}
Indeed, \eqref{CLAIM} follows immediately from Lemma~\ref{dec} applied for $n=n_j$ and with $\psi_n=\mathbf 1_{\mathcal C_{\sigma^n \omega}}$, by passing to the limit when $j\to \infty$. Using the assumption in the statement of the proposition together with the density of continuous functions in $L^1(\mu_\omega)$, one can conclude that~\eqref{CLAIM} holds also for $\varphi \in L^1(\mu_\omega)$. This yields that $\mu_\omega(\mathcal C_\omega)=1$ for $\omega \in \Omega$, and consequently $\mu(\mathcal C)=1$.
\end{proof}

\bibliographystyle{amsplain}

\end{document}